\documentclass{elsarticle}
\usepackage{animate}
\usepackage{amsmath}
\usepackage{amssymb}
\usepackage{amsthm}
\usepackage{url}
\usepackage{mathtools,hyperref}
\usepackage{graphics}
\usepackage{graphicx}
\usepackage{textcomp}
\usepackage{mathrsfs}
\usepackage{epstopdf}
\usepackage{algorithmic, algorithm}
\usepackage{braket,amsfonts,amsopn}
\usepackage{color}
\usepackage{accents}
\usepackage{mathscinet}
\usepackage{enumitem}
\usepackage{tikz}
\usepackage{cleveref}
\usepackage{lineno,hyperref}
\modulolinenumbers[5]
\usepackage{multirow}
\usepackage[T1]{fontenc}
\usepackage{textcomp}
\usepackage[scaled]{helvet}
\usepackage{subcaption}

\usepackage{stmaryrd}

\usepackage{caption}
\newtheorem{theorem}{Theorem}[section]

\theoremstyle{definition}
\newtheorem{definition}[theorem]{Definition}

\newtheorem{corollary}[theorem]{Corollary}
\newtheorem{assumption}[theorem]{Assumption}
\newtheorem{remark}{Remark}[section]

\numberwithin{equation}{section}

\usepackage{xparse}

\NewDocumentCommand{\dgal}{sO{}m}{%
  \IfBooleanTF{#1}
    {\dgalext{#3}}
    {\dgalx[#2]{#3}}%
}

\NewDocumentCommand{\dgalext}{m}{%
  \sbox0{%
    \mathsurround=0pt 
    $\left\{\vphantom{#1}\right.\kern-\nulldelimiterspace$%
  }%
  \sbox2{\{}%
  \ifdim\ht0=\ht2
    \{\kern-.45\wd2 \{#1\}\kern-.45\wd2 \}%
  \else
    \left\{\kern-.5\wd0\left\{#1\right\}\kern-.5\wd0\right\}%
  \fi
}

\NewDocumentCommand{\dgalx}{om}{%
  \sbox0{\mathsurround=0pt$#1\{$}%
  \sbox2{\{}%
  \ifdim\ht0=\ht2
    \{\kern-.45\wd2 \{#2\}\kern-.45\wd2 \}%
  \else
    \mathopen{#1\{\kern-.5\wd0 #1\{}
    #2
    \mathclose{#1\}\kern-.5\wd0 #1\}}
  \fi
}

\usepackage{natbib,hyperref}
\begin{document}

\begin{frontmatter}

\title{Gradient recovery for elliptic interface problem: III. Nitsche's method}

\author[mainaddress]{Hailong Guo}
\ead{hlguo@math.ucsb.edu}

\author[mainaddress]{Xu Yang\corref{mycorrespondingauthor}}
\cortext[mycorrespondingauthor]{Corresponding author}
\ead{xuyang@math.ucsb.edu}

\address[mainaddress]{Department of Mathematics, University of California Santa Barbara, CA, 93106}

\begin{abstract}
This is the third paper on the study of gradient recovery for elliptic interface problem. In our previous works [H. Guo and X. Yang, 2016, arXiv:1607.05898 and {\it J. Comput. Phys.}, 338 (2017), 606--619], we developed {gradient recovery methods} for elliptic interface problem based on body-fitted meshes and immersed finite element methods. Despite the efficiency and accuracy that these methods bring to recover the gradient, there are still some cases in unfitted meshes where skinny triangles appear in the generated local body-fitted  triangulation  that destroy the accuracy of recovered gradient near the interface. In this paper, we propose a gradient recovery technique based on Nitsche's method for elliptic interface problem, which avoids the loss of accuracy of gradient near the interface caused by skinny triangles. We analyze the supercloseness between the gradient of  the numerical solution by the Nitsche's method and the gradient of interpolation of the exact solution, which leads to the superconvergence of the proposed gradient recovery method. We also present several numerical examples to validate the theoretical results.
\end{abstract}


\end{frontmatter}


\section{Introduction}
\label{sec:int}
Elliptic interface problems arise in many applications such as fluid dynamics and materials science,
where the background consists of rather different materials on the subdomains separated by smooth curves called interface.
The numerical challenge of interface problems comes from the fact that the solution, in general, has low global regularity due to the discontinuity of parameter (e.g. dielectric constant) at the interface. Standard finite element methods have been studied for elliptic interface problems by aligning the triangulation along the interface (body-fitted meshes), and are proven to achieve optimal convergence rates in both  $L_2$ and energy norms  \cite{Bramble1996,Babuska1970,ChenZou1998,Xu1982}. However, when  the interface leads to subdomains of complex geometry, it is non-trivial and time-consuming to generate body-fitted meshes.

To overcome the difficulty of mesh generation in standard finite element method, tremendous effort has been input to develop numerical methods using unfitted (Cartesian) meshes. \cite{Peskin1977} is the first to propose the immersed boundary method (IBM) to simulate blood
flow using Cartesian meshes. The idea of IBM is to use a Dirac $\delta$-function to model the discontinuity and discretize it to distribute a singular source to nearest grid points \cite{Peskin1977, Peskin2002}. But IBM only achieves the first-order accuracy. To improve the accuracy, Leveque and Li characterized the discontinuity as jump conditions and proposed the immersed interface method (IIM) \cite{LevequeLi1994}. IIM constructs special finite difference schemes to incorporate the jump conditions near the interface.
High order unfitted finite difference methods including matched interface and boundary (MIB) method are also proposed in \cite{Zhou2006, ZhouWei2006}.
We refer to \cite{LiIto2006} for a review on IIM and other unfitted finite difference methods.

In the meantime, unfitted numerical methods using finite element formulation are also developed for elliptic interface problems. The extended finite element method \cite{Belytschko1999,Moes2001,Fries2010} enriches the standard continuous  finite element
space by adding some special basis functions to capture the discontinuity.  The immersed finite element methods  \cite{Li1998, LiLinWu2003,LiLinLin2004}  modify the basis functions to satisfy the homogeneous jump conditions  on interface elements.  For the nonconforming immersed finite element method(IFEM) in \cite{LiLinWu2003}, the numerical solution is continuous inside  each element but
 can be discontinuous on the boundary of each element. Recently, there are also improved versions of IFEM such as the Petrov-Galerkin IFEM \cite{HouLiu2005,HouWuZhang2004, HouWangWang2010}, symmetric and consistent IFEM \cite{JiChenLi2014},
and partially penalized IFEM \cite{LinLinZhang2015}.

The Nitsche's method \cite{Roland2009, Burman2014, Burman2015, Burman2012, Hansbo2005,Hansbo2002, Hansbo2004, Hansbo2003, Hansbo2014}, also called the cut finite element method, is firstly proposed by Hansbo and Hansbo in \cite{Hansbo2002} to solve elliptic interface problems using unfitted meshes.  It is further extended to deal with elastic problems with  strong  and weak discontinuities \cite{Hansbo2004}.
The study of the Nitsche's method for Stokes interface problems can be found in \cite{Hansbo2014}.
The key idea of the Nitsche's method
is to construct an approximate solution on each fictitious domain and  use Nitsche's technique \cite{Nitsche1971} to patch them together.
A similar idea was used to develop the fictitious domain method \cite{Burman2012, Burman2014}.
The robust forms of the unfitted  Nitsche's method were given in \cite{Dolbow2012, Wadbro2013}.   The recent development of the cut finite element method is referred
to the review paper \cite{Burman2015}.

For elliptic interface problems, computation of gradient plays an important role in many practical problems as discussed in \cite{Lizhilin2017}, which demands numerical methods of high order accuracy. For standard elliptic problems, it is well known that the gradient recovery techniques  \cite{ZZ1992a, ZTZ2013,ZZ1992b, ZhangNaga2005, GuoZhang2015, AinsworthOden2000, NagaZhang2005, NagaZhang2004,ChenXu2007, XuZhang2004, BankXu2003b}
can reconstruct a highly accurate approximate gradient from the primarily computed data with reasonable cost.  But for elliptic interface problems, only a few works have been done on the gradient recovery and associated superconvergence theory. For example in \cite{Chen2014},  a supercloseness result between the gradient of the linear finite element solution and the gradient of the linear interpolation is proved for a two-dimensional interface problem with a body-fitted mesh.  For IFEM, Chou et al. introduced two special interpolation formulae to recover flux with high order accuracy for the one-dimensional linear and quadratic IFEM
\cite{Chou2012,Chou2015}. Moreover, Li and his collaborators recently proposed an augmented immersed interface method \cite{Lizhilin2017} and a new finite element method \cite{Qin2017} to accurately compute the gradient of the solution to elliptic interface problems. In our recent work \cite{GuoYang2016},
 we proposed an improved polynomial preserving recovery for elliptic interface problems based on a body-fitted mesh and proved the superconvergence on both mildly unstructured meshes and adaptively refined meshes. Later in the two-dimensional case \cite{GuoYang2017}, we  proposed gradient recovery methods based on symmetric and consistent IFEM \cite{JiChenLi2014}
and Petrov-Galerkin  IFEM \cite{HouLiu2005,HouSongWangZhao2013,HouWuZhang2004} and numerically verified its superconvergence. In \cite{GuoYangZhang2017}, we also provided a supercloseness result for the partially penalized IFEMs and proved that the recovered gradient using the gradient recovery method in \cite{GuoYang2017} is superconvergent to the exact gradient.


Despite the efficiency and accuracy that the methods mentioned above bring to recover the gradient, there are still some cases  in unfitted meshes where skinny triangles appear in the generated local body-fitted  triangulation that destroy the accuracy of recovered gradient near the interface. In this paper, we propose an unfitted polynomial preserving recovery (UPPR) based on the Nitsche's method. The key idea is to decompose the domain
 into two overlapping subdomains, named fictitious domains, by the interface and the triangulation proposed in the Nitsche's method.
On each fictitious domain, the standard linear finite element space will be used, and thus the classical polynomial preserving recovery (PPR) can be applied in each fictitious domain. Compared to previous gradient recovery methods \cite{GuoYang2016, GuoYang2017},  the new method does not require generating a local body-fitted mesh  and therefore avoids the drawback caused by skinny triangles. In general, the exact solutions of the interface problems are piecewise smooth on each subdomain.
It implies that the extension of the exact solution on  each subdomain to the whole domain is smooth, based on which,  the   recovered gradient using the interpolation of the exact solution is proven to be superconvergent to the exact gradient at rate of $\mathcal{O}(h^2)$, and this is similar to
  the classical PPR for standard elliptic problems.   In addition, we prove  $\mathcal{O}(h^{1.5})$ supercloseness between the gradient given by the Nitsche's method  and the gradient of the interpolation of the exact solution by a sharp argument. This enables us to establish the complete superconvergence theory for the proposed UPPR.

The rest of the paper is organized as follows:
We introduce briefly the elliptic interface problem  and the  Nitsche's method  in Section 2. In Section 3, we analyze the supercloseness  for the Nitsche's method and prove the $\mathcal{O}(h^{1.5})$ supercloseness between the gradient of the finite element solution  and the gradient of the interpolation of the exact solution.
In section 4, we describe the UPPR for the Nitsche's method and establish its superconvergence theory.
 In Section 5, we present several numerical examples to confirm our theoretical results.

\section{Nitsche's method for elliptic interface problem}
In this section, we first introduce the elliptic interface problem and associated notations, and then summarize
the unfitted finite element discretization based on Nitsche's method proposed in \cite{Dolbow2012, Hansbo2002} as a preparation for the unfitted polynomial preserving recovery (UPPR) method introduced later.
\label{sec:pl}
\subsection{Elliptic interface problem}
\label{ssec:interface}

Let $\Omega$ be a bounded polygonal domain with  Lipschitz boundary $\partial \Omega$ in $\mathbb{R}^2$.
A $C^2$-curve $\Gamma$ divides $\Omega$ into two disjoint subdomains $\Omega_1$ and $\Omega_2$ as in Figure \ref{fig:interface}.
We consider the following elliptic interface problem
\begin{subequations}\label{eq:interface}
 \begin{align}
  -\nabla \cdot (\beta(z) \nabla u(z)) &= f(z),  \quad \text{ in } \Omega_1\cup\Omega_2, \label{eq:model}\\
   u & = 0, \quad\quad\,\,  \text{ on } \partial\Omega, \label{eq:bnd}\\
      \llbracket u\rrbracket &=q,  \quad\quad\,\,  \text{ on } \Gamma \label{eq:valuejump}\\
        \llbracket \beta \partial_n u \rrbracket&= g,   \quad\quad\,\,  \text{ on } \Gamma\label{eq:fluxjump}
\end{align}
\end{subequations}
where  $\partial_n u=(\nabla u)\cdot n$ with $n$ being the unit outward normal vector of $\Gamma$   and  the
jump $\llbracket w \rrbracket$ on $\Gamma$  is defined as
\begin{equation}\label{eq:jump}
\llbracket w\rrbracket = w_1 - w_2,
\end{equation}
with $w_i = w|_{\Omega_i}$ being the restriction of $w$ on $\Omega_i$.
The diffusion coefficient $\beta(z) \ge \beta_0$ is a piecewise smooth function, i.e.
\begin{equation}
\beta(z) =
\left\{
\begin{array}{ccc}
    \beta_1(z) &  \text{if } z= (x,y)\in \Omega_1, \\
   \beta_2(z)  &   \text{if } z= (x,y)\in \Omega_2,\\

\end{array}
\right.
\end{equation}
which has a finite jump of function value at the interface $\Gamma$.

\begin{figure}[ht]
    \centering
    \includegraphics[width=0.5\textwidth]{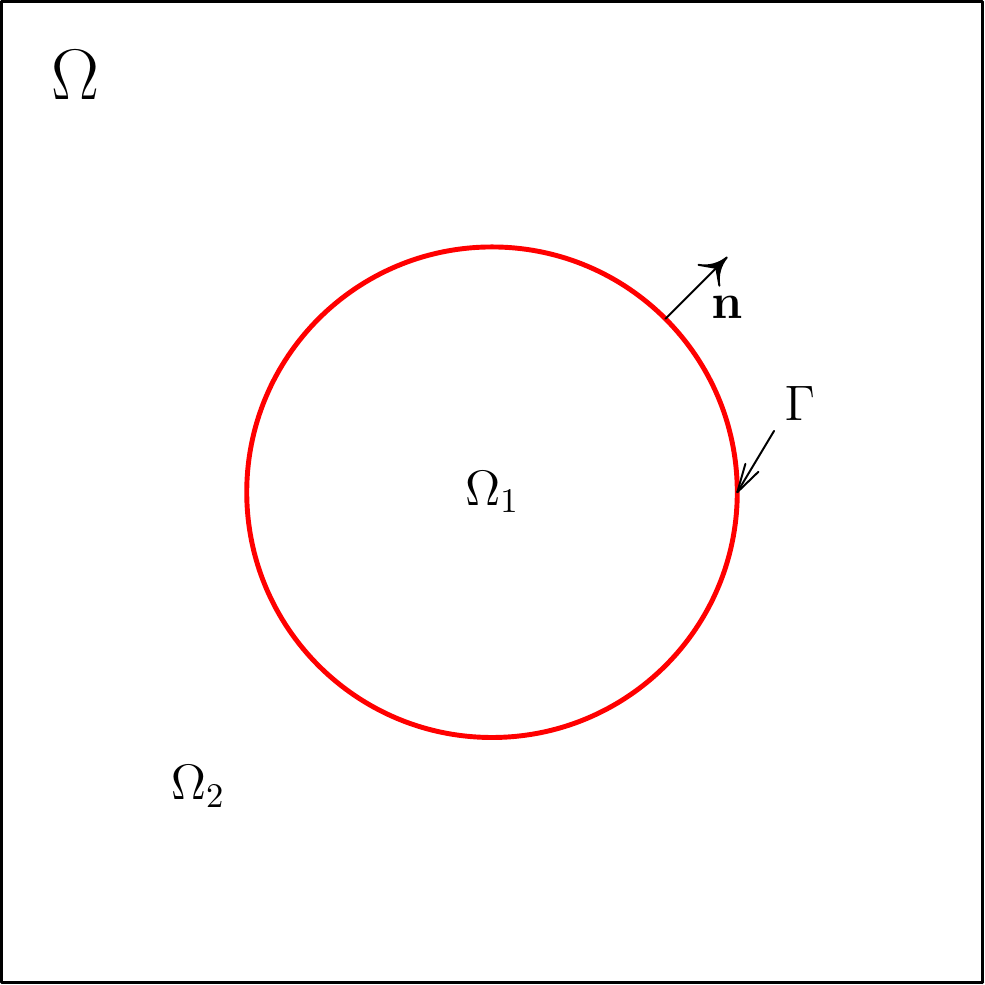}
    \caption{ Typical example of domain $\Omega$ with interface $\Gamma$.}
\label{fig:interface}
\end{figure}

In this paper,  we use the standard
notations for Sobolev spaces and their associated norms as in \cite{BrennerScott2008, Ciarlet2002, Evans2008}.
For any  bounded domain $D\subset \Omega$, the Sobolev space with norm
$\|\cdot\|_{k, p, D} $ and seminorm $|\cdot|_{k, p,D}$ is denoted by $W^{k,p}(D)$.
 When $p = 2$, $W^{k,2}(D)$ is simply  denoted by $H^{k}(D)$
 and the subscript $p$ is omitted in its associate norm and seminorm.  Similar notations are applied to subdomains  of $\Gamma$.
 Let $(\cdot, \cdot)_D$ and $\langle \cdot, \cdot\rangle_{\Gamma}$ denote the standard $L_2$ inner products of $L_2(D)$ and
 $L_2(\Gamma)$, respectively.
For a bounded domain $D= D_1\cup D_2$ with $D_1\cap D_2=\emptyset$, let  $W^{k,p}(D_1\cup D_2)$  be the function space consisting of piecewise Sobolev  functions $w$  such
 that $w|_{D_1}\in W^{k,p}(D_1)$ and $w|_{D_2}\in W^{k,p}(D_2)$, whose  norm  is defined as
 \begin{equation}\label{eq:pnorm}
\|w\|_{k,p, D_1\cup D_2} = \left( \|w\|_{k,p, D_1}^p + \|w\|_{k,p, D_2}^p\right)^{1/p},
\end{equation}
and seminorm is defined as
 \begin{equation}\label{eq:psnorm}
|w|_{k,p, D_1\cup D_2} = \left( |w|_{k,p, D_1}^p + |w|_{k,p, D_2}^p\right)^{1/p}.
\end{equation}

In this paper, we denote $C$ as a generic positive constant which can be different at different occurrences.
In addition,  it is  independent of mesh size and the location of the interface.

\subsection{Nitsche's method}
\label{ssec:nit}
Let $\mathcal{T}_h$ be a  triangulation of $\Omega$ independent of the location of  the interface $\Gamma$.
For any element $T\in \mathcal{T}_h$, let $h_T$ be the diameter of $T$ and $\rho_T$ be the diameter of the circle
inscribed in $T$.  In addition, we make the following assumptions on the triangulation.

\begin{assumption}\label{ass:reg}
The triangulation $\mathcal{T}_h$ is shape regular in the sense that there is a constant $\sigma$ such that
\begin{equation}
\frac{h_T}{\rho_T}\le \sigma,
\end{equation}
for any $ T\in \mathcal{T}_h$.
\end{assumption}
\begin{assumption}\label{ass:inter}
The interface $\Gamma$ intersects each  interface element boundary $\partial T$ exactly twice, and each open edge at most once.
\end{assumption}


To define the finite element space,  denote the set of all elements that intersect the interface $\Gamma$ by
\begin{equation}
\mathcal{T}_{\Gamma,h} = \left\{ T\in \mathcal{T}_h:  \Gamma\cap \overline{T} \neq \emptyset \right\},
\end{equation}
and denote the union of all such type elements by
\begin{equation}
\Omega_{\Gamma,h} = \bigcup\limits_{T\in \mathcal{T}_{\Gamma,h}}T.
\end{equation}
Denote the set of all elements covering subdomain $\Omega_i$ to be
\begin{equation}
\mathcal{T}_{i,h} = \left\{ T\in \mathcal{T}_h:   \overline{\Omega_i}\cap \overline{T} \neq \emptyset\right\}, \quad i = 1, 2;
\end{equation}
and let
\begin{equation}
\Omega_{i,h} = \bigcup\limits_{T\in \mathcal{T}_{i,h}}T, \quad \omega_{i,h} = \bigcup\limits_{T\in \mathcal{T}_{i,h}\setminus \mathcal{T}_{\Gamma,h}}T,\quad i = 1, 2.
\end{equation}
Figure \ref{fig:mesh} gives an illustration of $\Omega_{i,h}$ and $\omega_{i,h}$.  We remark that
$\Omega_{1,h}$ and $\Omega_{2,h}$ overlap on $\Omega_{\Gamma,h}$, which is shown as the shaded part in Figures \ref{fig:meshone} and \ref{fig:meshtwo}.

\begin{figure}
   \centering
   \subcaptionbox{\label{fig:meshwhole}}
  {\includegraphics[width=0.32\textwidth]{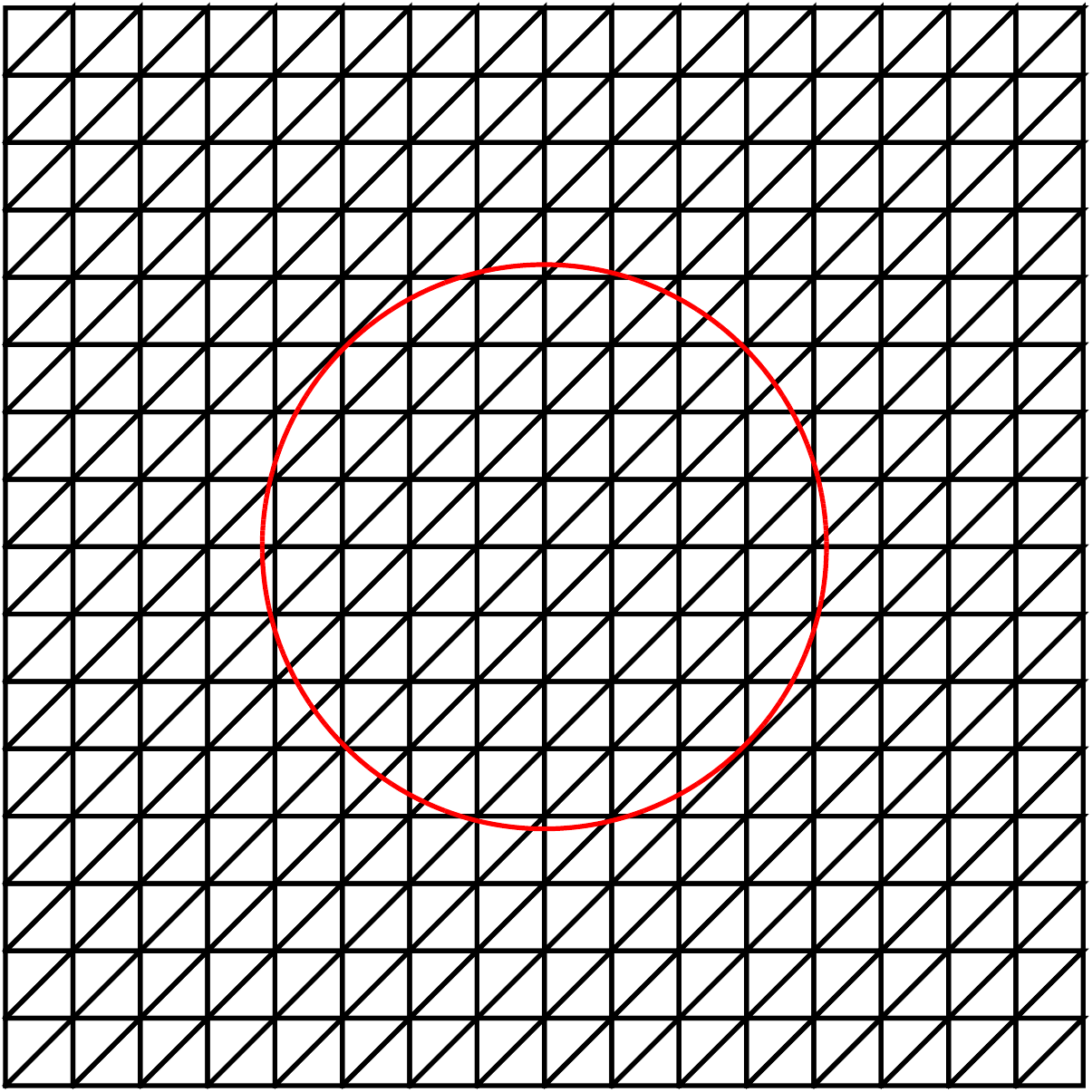}}
  \subcaptionbox{\label{fig:meshone}}
   {\includegraphics[width=0.32\textwidth]{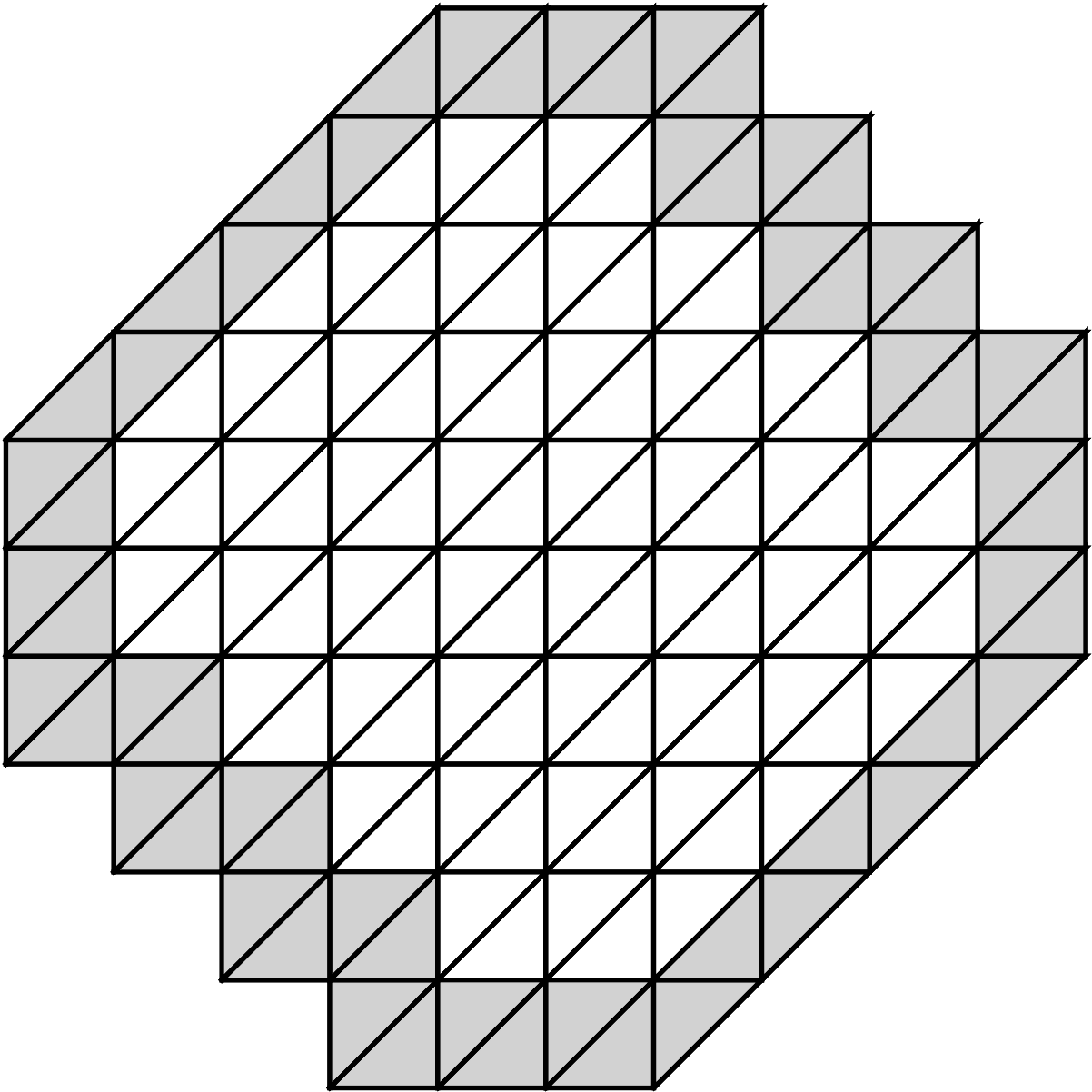}}
  \subcaptionbox{\label{fig:meshtwo}}
  {\includegraphics[width=0.32\textwidth]{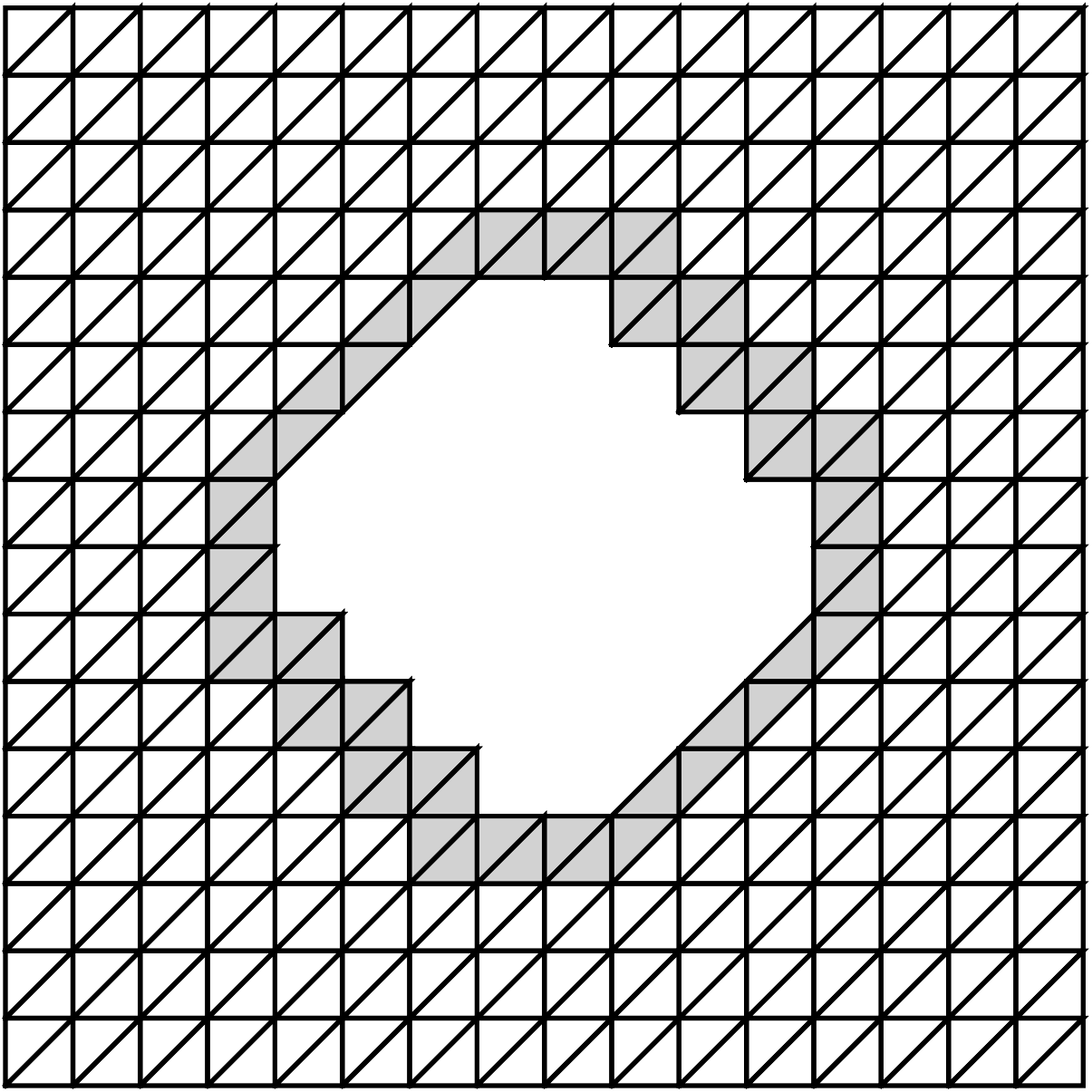}}
   \caption{Triangulat $\mathcal{T}_h$ on the square domain $\Gamma$ with circular interface $\Gamma$. (a): Triangulation $\mathcal{T}_h$; (b): Triangulation $\mathcal{T}_{1,h}$
   and $\omega_{1,h}$ (non-shaded triangles);
   (c): Triangulation $\mathcal{T}_{2,h}$ and $\omega_{2,h}$ (non-shaded triangles).}\label{fig:mesh}
\end{figure}

Let $V_{i,h}$ be the standard continuous linear finite element space on $\Omega_{i,h}$, i.e.
\begin{equation}
V_{i,h} = \left\{ v \in C^0(\Omega_{i,h}):  v|_{T} \in \mathbb{P}_1(T) \text{ for any } T \in \mathcal{T}_{i,h}\right\},\quad i = 1, 2,
\end{equation}
where $\mathbb{P}_1(T)$ is the space of polynomials with degree $\leq 1$ on $T$.  Then,  we define
the finite element space $V_h$ as
\begin{equation}
V_h = \left\{ v_h = (v_{1,h}, v_{2,h}): v_{i,h}\in V_{i,h}, \, i = 1,2 \right\},
\end{equation}
and $V_{h,0}$ as
\begin{equation}
V_{h,0} = \left\{ v_h\in V_h: v_h|_{\partial\Omega} = 0\right\}.
\end{equation}
Note that a function in $V_h$ is a vector-valued function from $\mathbb{R}^2\mapsto\mathbb{R}^2$, which has a zero component in $\omega_{1,h}\bigcup\omega_{2,h}$ but in general two non-zero components in $\mathcal{T}_{\Gamma,h}$.  It means that one will have two sets
of basis functions for any element $T$ in $\mathcal{T}_{\Gamma,h}$: one for $V_{1,h}$ and the other for $V_{2,h}$.

For any $T$ in $\mathcal{T}_{\Gamma,h}$,  denote $T_i = T \cap \Omega_i$ as the part of $T$ in $\Omega_i$, with
$|T_i|$ being the measure of $T_i$ in $\mathbb{R}^2$.
Denote $\Gamma_T = \Gamma\cap T$ as the part of $\Gamma$ in $T$, with $|\Gamma_T|$ being the measure of $\Gamma_T$ in $\mathbb{R}^1$.
To increase the robustness of the Nitsche's method, we  introduce two weights as used in \cite{Dolbow2012}
\begin{equation}
\kappa_1|_T = \frac{\beta_2|T_1|}{\beta_2|T_1|+\beta_1|T_2|}, \quad \kappa_2|_T = \frac{\beta_1|T_2|}{\beta_2|T_1|+\beta_1|T_2|},
\end{equation}
which satisfies that $ \kappa_1 + \kappa_2=1$.
Then,  we define the weighted  averaging of a function $v_h$ in $V_h$ on the interface $\Gamma$ as
\begin{equation}
\dgal{w} = \kappa_1v_{1,h} +\kappa_2v_{2,h}, \quad  \dgal{w} ^{\ast}= \kappa_2v_{1,h} +\kappa_1v_{2,h}.
\end{equation}

Using the notations introduced above, the Nitsche's method \cite{Burman2015,Hansbo2002,Dolbow2012} for the elliptic interface problem \eqref{eq:interface}  is  to find $u_h \in V_{h,0}$ such that
\begin{equation}\label{eq:var}
a_h(u_h, v_h) = L_h(v_h), \quad \forall v_h \in V_{h,0};
\end{equation}
where the bilinear form $a_h$ is defined as
\begin{equation}\label{eq:bilinear}
\begin{split}
 a_h(u_h,v_h) = &\sum\limits_{i=1}^2\left(\beta\nabla u_{i,h}, \nabla v_{i,h}\right)_{\Omega_i}
- \left\langle  \llbracket u_h \rrbracket ,\dgal{\beta \partial_nv_{h}}\right\rangle_{\Gamma} \\
&-\left\langle  \llbracket v_h \rrbracket ,\dgal{\beta \partial_nu_{h}}\right\rangle_{\Gamma}
+ h^{-1}\left\langle \gamma \llbracket u_h \rrbracket,  \llbracket v_h \rrbracket\right\rangle_{\Gamma},
\end{split}
\end{equation}
and the linear functional $L_h$ is defined as
\begin{equation}\label{eq:functional}
L_h(v_h) = \sum\limits_{i=1}^2(f, \nabla v_{i,h})_{\Omega_i} - \langle q,\dgal{\beta\partial_n v_h}\rangle_{\Gamma}
+ \left\langle \gamma q,  \llbracket v_h \rrbracket , \right\rangle_{\Gamma}
+\left\langle g, \dgal{w}^{\ast}\right\rangle_{\Gamma},
\end{equation}
with the stability parameter
\begin{equation}
\gamma|_T = \frac{2h_T|\Gamma_T|}{|T_1|/\beta_1+|T_2|/\beta_2}.
\end{equation}

In \cite{Dolbow2012},  the discrete  variational form is shown to be consistent  as the following theorem:
\begin{theorem}\label{thm:consist}
Let $u$ be the solution of the interface problem \eqref{eq:interface}. Then we have
\begin{equation}\label{eq:nitsche}
a_h(\tilde{u}, v_h) = L(v_h), \quad \forall v_h \in V_{h,0},
\end{equation}
where $\tilde{u}=(u_1, u_2)$ with the short-hand notation $u_i=u|_{\Omega_i}$ as used in \eqref{eq:jump}.
\end{theorem}

Theorem \ref{thm:consist} implies the following Galerkin orthogonality:
\begin{corollary}\label{cor:orth}
 Let $u$ be the solution of \eqref{eq:interface} and $u_h$ be the solution of the discrete problem \eqref{eq:var}. Then we have
 \begin{equation}\label{eq:orth}
a_h(\tilde{u}-u_h, v_h) = 0, \quad \forall v_h \in V_{h,0},
\end{equation}
with $\tilde{u}=(u_1, u_2)$.
\end{corollary}

To analyze the stability  of the bilinear form $a_h(\cdot, \cdot)$, we introduce the following mesh-dependent norm \cite{Burman2015,Hansbo2002}
\begin{equation}\label{def:md_norm}
|||v|||_h^2 = \|\nabla u\|_{0, \Omega_1\cup \Omega_2}^2 + \sum\limits_{T\in \mathcal{T}_{\Gamma,h}}h_T\|\dgal{\partial_n v}\|_{0,\Gamma_T}^2
+\sum\limits_{T\in \mathcal{T}_{\Gamma,h}}h_T^{-1}\|\llbracket v\rrbracket\|_{0,\Gamma_T}^2.
\end{equation}

In \cite{Hansbo2002},  it is shown that the bilinear form $a_h(\cdot, \cdot)$  is coercive  with respect to the above mesh-dependent norm in the following sense
\begin{theorem}\label{thm:coer}
There is a constant C such that
  \begin{equation}
C|||v_h|||_h^2 \le a_h(v_h,v_h), \quad \forall v_h \in V_h.
\end{equation}
\end{theorem}

Based on the above coercivity, Hansbo et al. proved the following optimal convergence result \cite{Hansbo2002}:
\begin{theorem}\label{thm:optimal}
 Let $u$ be the solution to the interface problem \eqref{eq:interface} and $u_h$ be the finite element solution to \eqref{eq:var}.
 If $u \in H^2(\Omega_1\cup \Omega_2)$, then
 \begin{equation}
|||\tilde{u}-u_h|||_h\le Ch \|u\|_{2, \Omega_1\cup\Omega_2},
\end{equation}
and
  \begin{equation}
||\tilde{u}-u_h||_{0, \Omega}\le Ch^2 \|u\|_{2, \Omega_1\cup\Omega_2},
\end{equation}
with $\tilde{u}=(u_1, u_2)$.
\end{theorem}
\section{Supercloseness Analysis}\label{sec:superclose}
In this section, we establish the supercloseness result between the gradient of the finite element solution and the gradient of the interpolation of the exact solution as a preparation for the superconvergence anaysis of the proposed gradient recovery technique based on Nitsche's method. For that propose, we need the triangulation $\mathcal{T}_h$ to satisfy Condition $(\sigma,\alpha)$ as explained below.

Two adjacent triangles are said to  form an $\mathcal{O}(h^{1+\alpha})$  approximate parallelogram if the lengths of any two opposite edges differ only by $\mathcal{O}(h^{1+\alpha})$.

\begin{definition}
 The triangulation $\mathcal{T}_h$ is called to satisfy Condition
$(\sigma,\alpha)$ if
there exist a partition $\mathcal{T}_{h}^{1} \cup \mathcal{T}_{h}^{2}$ of $\mathcal{T}_h$
and positive constants $\alpha$ and $\sigma$ such that every
two adjacent triangles in $\mathcal{T}_{h}^{1}$ form an $\mathcal{O}(h^{1+\alpha})$ parallelogram and
$$
\sum_{T\in {\mathcal{T}_{h}^{2}}} |T| = \mathcal{O}(h^\sigma).
$$
\end{definition}
\vspace{-0.15in}
\begin{remark}
For the Nitsche's method, we usually use a Cartesian mesh which is independent of the location of the interface.
Therefore any Cartesian mesh satisfies  Condition $(\sigma,\alpha)$ with $\sigma = \infty$ and $\alpha = 1$.
\end{remark}

To define the interpolation operator, we need  to extend the function defined on the subdomain $\Omega_i$ to the whole domain $\Omega$.
Let $E_i$, $i=1,2$, be the $H^3$-extension operator from $H^3(\Omega_i)$ to $H^3(\Omega)$ such that
\begin{equation}
(E_iw)|_{\Omega_i} = w,
\end{equation}
and
\begin{equation}
\|E_iw\|_{s, \Omega} \le C \|w\|_{s, \Omega_i}, \quad \forall w\in H^s(\Omega_i), \, s = 0, 1, 2,3.
\end{equation}
Let $I_{i,h}$ be the standard nodal interpolation operator  from $C(\overline{\Omega})$ to $V_{i,h}$.
Define the interpolation operator for the finite element space $V_h$ as
\begin{equation}\label{eq:feminterp}
I_h^{\ast}v = (I_{1,h}^{\ast}v_1, I_{1,h}^{\ast}v_2),
\end{equation}
where
\begin{equation}\label{eq:feminterpp}
I_{i,h}^{\ast} = I_{i,h}E_iv_i, \, i = 1,2.
\end{equation}

Optimal approximation capability of $I_h^{\ast}$ is proved in \cite{Hansbo2002}.
Assume  $\mathcal{T}_h$ satisfies  Condition $(\sigma,\alpha)$, and then we can prove the following theorem:
\begin{theorem}\label{thm:supercloseness}
 Suppose the triangulation $\mathcal{T}_h$ satisfies  Condition $(\sigma,\alpha)$.  Let $u$ be the solution of
 the interface problem  \eqref{eq:interface} and $I_h^{\ast}u$ be the interpolation of $u$ in the finite element space
 $V_{h,0}$. If $u \in H^1(\Omega) \cap H^{3}(\Omega^-\cup\Omega^+)\cap W^{2, \infty}(\Omega^-\cup\Omega^+)$,
 then for $\tilde{u}=(u_1, u_2)$ and all $v_h \in V_{h,0}$,
 \begin{multline}\label{eq:superclose}
a_h(\tilde{u}-I_h^{\ast}u, v_h) \le C\biggl(h^{1+\rho}(\|u\|_{3, \Omega^+\cup\Omega^-} + \|u\|_{2, \infty, \Omega^+\cup\Omega^-}) \\ + Ch^{3/2} \|u\|_{2, \infty, \Omega^+\cup\Omega^-}\biggr)|v_h|_{h}.
\end{multline}
where  $\rho = \min(\alpha, \frac{\sigma}{2}, \frac{1}{2})$.
\end{theorem}
\begin{proof}
 By \eqref{eq:bilinear}, we have
 \begin{equation}
\begin{aligned}
 &a_h(\tilde{u}-I_h^{\ast}u, v_h) \\
 = &\sum\limits_{i=1}^2\left(\beta\nabla (u_i-I_{i,h}^{\ast}u_i), \nabla v_{i,h}\right)_{\Omega_i}
- \left\langle  \llbracket u-I_h^{\ast}u \rrbracket ,\dgal{\beta \partial_nv_{h}}\right\rangle_{\Gamma} \\
&-\left\langle  \llbracket v_h \rrbracket ,\dgal{\beta \partial_n(u-I_h^{\ast} u)}\right\rangle_{\Gamma}
+h^{-1}\left\langle \gamma  \llbracket u-I_h^{\ast} u \rrbracket,  \llbracket v_h \rrbracket\right\rangle_{\Gamma}\\
 = &  \left(\beta\nabla (u_1-I_{1,h}^{\ast}u_1), \nabla v_{1,h}\right)_{\omega_{1,h}} + \left(\beta\nabla (u_2-I_{2,h}^{\ast}u_2), \nabla v_{2,h}\right)_{\omega_{2,h}}\\
 & +\left(\beta\nabla (u-I_{h}^{\ast}u), \nabla v_{h}u\right)_{\Omega_{\Gamma,h}}  - \left\langle  \llbracket u-I_h^{\ast}u \rrbracket ,\dgal{\beta \partial_nv_{h}}\right\rangle_{\Gamma} \\
 &-\left\langle  \llbracket v_h \rrbracket ,\dgal{\beta \partial_n(u-I_h^{\ast} u)}\right\rangle_{\Gamma}
 + h^{-1}\left\langle \gamma  \llbracket u-I_h^{\ast} u \rrbracket,  \llbracket v_h \rrbracket\right\rangle_{\Gamma}\\
 :=&F_1+F_2+F_3+F_4+F_5+F_6.
\end{aligned}
\end{equation}
Since $\mathcal{T}_h$  satisfies  Condition $(\sigma, \alpha)$, it follows that $\mathcal{T}_{1,h}\setminus\mathcal{T}_{\Gamma,h}$ and
$\mathcal{T}_{2,h}\setminus\mathcal{T}_{\Gamma,h}$also satisfy  Condition  $(\sigma, \alpha)$.  Notice that the restriction of the finite element
space $V_{i,h}$ on $\omega_{i,h}$ is just the standard continuous linear finite element space for $i = 1, 2$.  Then by Lemma $2.1$ in \cite{XuZhang2004}, we have
\begin{align}
 |F_1|\le Ch^{1+\rho}(\|u\|_{3, \Omega_1} + \|u\|_{2, \infty, \Omega_1})|v_h|_{h},\\
  |F_2|\le Ch^{1+\rho}(\|u\|_{3, \Omega_2} + \|u\|_{2, \infty, \Omega_2})|v_h|_{h},
\end{align}
where  $\rho = \min(\alpha, \frac{\sigma}{2}, \frac{1}{2})$.
To estimate $I_3$,  the Cauchy-Schwartz inequality implies
\begin{equation*}
\begin{aligned}
  F_3  \le &C_1 \left(\sum_{T\in \mathcal{T}_{\Gamma,h}} \|\nabla (u_1-I_{1,h}^{\ast}u_1)\|_{0,T_1}^2\right)^{1/2} \left(\sum_{T\in \mathcal{T}_{\Gamma,h}}\| \nabla v_{1,h}\|_{0,T_1}^2\right)^{1/2}  \\
&+C_2 \left(\sum_{T\in \mathcal{T}_{\Gamma,h}} \|\nabla (u_2-I_{1,h}^{\ast}u_2)\|_{0,T_2}^2\right)^{1/2} \left(\sum_{T\in \mathcal{T}_{\Gamma,h}}\| \nabla v_{2,h}\|_{0,T_2}^2\right)^{1/2}
\end{aligned}
\end{equation*}
\begin{equation*}
\begin{aligned}
    \le &C_1 \left(\sum_{T\in \mathcal{T}_{\Gamma,h}} h^4\|u\|_{2,\infty,\Omega_1}^2\right)^{1/2} \left(\sum_{T\in \mathcal{T}_{\Gamma,h}}\| \nabla v_{1,h}\|_{0,T_1}^2\right)^{1/2}  \\
 & + C_2 \left(\sum_{T\in \mathcal{T}_{\Gamma,h}} h^4\|u\|_{2,\infty,\Omega_2}^2\right)^{1/2} \left(\sum_{T\in \mathcal{T}_{\Gamma,h}}\| \nabla v_{2,h}\|_{0,T_2}^2\right)^{1/2}\\
    \le &C_1 h^2\|u\|_{2,\infty,\Omega_1} \left(\sum_{T\in \mathcal{T}_{\Gamma,h}} 1\right)^{1/2} \left(\sum_{T\in \mathcal{T}_{\Gamma,h}}\| \nabla v_{1,h}\|_{0,T_1}^2\right)^{1/2}  \\
&+C_2 h^2\|u\|_{2,\infty,\Omega_2} \left(\sum_{T\in \mathcal{T}_{\Gamma,h}}1\right)^{1/2} \left(\sum_{T\in \mathcal{T}_{\Gamma,h}}\| \nabla v_{2,h}\|_{0,T_2}^2\right)^{1/2}\\
  \le &C_1 h^{3/2}\|u\|_{2,\infty,\Omega_1}|v|_{1,\Omega_1} +
C_2 h^{3/2}\|u\|_{2,\infty,\Omega_2} |v|_{1,\Omega_2} \\
\le& C h^{3/2}\|u\|_{2,\infty,\Omega_1\cup\Omega_2} |v|_{1,\Omega_1\cup\Omega_2} ;
\end{aligned}
\end{equation*}
where we have used the fact $ \sum_{T\in \mathcal{T}_{\Gamma,h}} 1 \approx \mathcal{O}(h^{-1})$.  By the Cauchy-Schwartz inequality and the trace inequality in \cite{Hansbo2002, Wadbro2013},  we have
\begin{equation}
\begin{split}
 F_4\le & \left(\sum_{T\in \mathcal{T}_{\Gamma,h}}h_T^{-1}  \|\llbracket u-I_h^{\ast}u \rrbracket\|_{0, \Gamma_T}^2\right)^{\frac{1}{2}}
 \left(\sum_{T\in \mathcal{T}_{\Gamma,h}}h_T\|\dgal{\beta \partial_nv_h}\|_{0,\Gamma_T}^2\right)^{\frac{1}{2}}\\
\le&  \left(\sum_{T\in \mathcal{T}_{\Gamma,h}}\sum_{i=1}^2h_T^{-1}  \|u_i-I_{i,h}^{\ast}u_i\|_{0, \Gamma_T}^2\right)^{\frac{1}{2}}|||v_h|||_h\\
\le&C\left(\sum_{T\in \mathcal{T}_{\Gamma,h}}\sum_{i=1}^2 \left( h_T^{-2}\|u_i-I_{i,h}^{\ast}u_i\|_{0, T}^2 +\|\nabla(u_i-I_{i,h}^{\ast}u_i)\|_{0, T}^2 \right)\right)^{\frac{1}{2}}|||v_h|||_h\\
\le& C\left(\sum_{T\in \mathcal{T}_{\Gamma,h}}\sum_{i=1}^2  h_T^{4}\|E_iu_i\|_{2,\infty ,T}^2 \right)^{\frac{1}{2}}|||v_h|||_h\\
\le& Ch^2\|u\|_{2,\infty, \Omega_1\cup\Omega_2}\left(\sum_{T\in \mathcal{T}_{\Gamma,h}}1 \right)^{\frac{1}{2}}|||v_h|||_h\\
\le& Ch^{3/2}\|u\|_{2,\infty, \Omega_1\cup\Omega_2}|||v_h|||_h.
\end{split}
\end{equation}
Similarly, we can estimate $F_5$ and $F_6$ as
\begin{equation}
\begin{split}
F_5 &\le Ch^{3/2}\|u\|_{2,\infty, \Omega_1\cup\Omega_2}|||v_h|||_h;\\
F_6 &\le  Ch^{3/2}\|u\|_{2,\infty, \Omega_1\cup\Omega_2}|||v_h|||_h.
\end{split}
\end{equation}
Combing all the above estimations, we get \eqref{eq:superclose}.
\end{proof}

Now we state our main supercloseness result as follows.
\begin{theorem}\label{thm:superinterp}
 Assume the same hypothesis as in Theorem \ref{thm:supercloseness} and let $u_h$ be the finite element solution of the discrete variational problem
 \eqref{eq:var}, then
 \begin{equation}\label{eq:superconvergence}
|||u_h-I_h^{\ast}u|||_h \le C\left(h^{1+\rho}(\|u\|_{3, \Omega^+\cup\Omega^-} + \|u\|_{2, \infty, \Omega^+\cup\Omega^-}) + h^{3/2} \|u\|_{2, \infty, \Omega^+\cup\Omega^-}\right),
\end{equation}
where $\rho = \min(\alpha, \frac{\sigma}{2}, \frac{1}{2})$.
\end{theorem}
\begin{proof}
 By Corollary \ref{cor:orth}, we have the following Galerkin orthogonality
 \begin{equation*}
a_h(\tilde{u}-u_h, v_h) = 0, \quad\forall v_h\in V_{h,0},
\end{equation*}
where $\tilde{u}=(u_1, u_2)$.
  Then we have
\begin{equation*}
a_h(u_h-I_h^{\ast}u, v_h)  = a_h(\tilde{u}-I_h^{\ast}u, v_h), \quad \forall v_h\in V_{h,0}.
\end{equation*}
Taking $v_h = u_h-I_h^{\ast}u$ and using Theorem \ref{thm:supercloseness} and Theorem \ref{thm:coer}, we prove \eqref{eq:superconvergence}.
\end{proof}

\begin{remark}
 Theorem \ref{thm:superinterp} implies that we can have the same supercloseness result as the partially penalized IFEM \cite{GuoYangZhang2017} .
\end{remark}

\section{Superconvergent gradient recovery}
In this section, we first propose the unfitted polynomial preserving recovery (UPPR) technique based on the Nitsche's method, then prove that the recovered gradient by UPPR is superconvergent to the exact gradient on mildly unstructured meshes.
\subsection{Unfitted polynomial preserving recovery}
To accurately recover the gradient,  we notice that the finite element solution $u_h$ of the Nitsche's method \eqref{eq:var} consists of two part: $u_{1,h}$ and $u_{2,h}$.
Also, by the fact that we can smoothly extend the exact solution $u|_{\Omega_i}$ ($i=1,2$) to the whole domain, it is safe to assume $u|_{\Omega_i}$ and its extension $E_iu|_{\Omega_i}$ is smooth in general.
For each $i \in \left\{1,2\right\}$, $u_{i,h}\in V_h$ is a continuous piecewise polynomial  on fictitious domain $\Omega_{i,h}$ but its gradient $\nabla u_{i,h}$ is only a piecewise constant function.
This motivates us to naturally  use some smoothing operators such as superconvergent patch recovery (SPR) and  polynomial preserving recovery (PPR) to smooth the discontinuous gradient into a continuous one on each fictitious domain $\Omega_{i,h}$.

To this end,  let $G_h^i$ be the PPR gradient recovery operator \cite{ZhangNaga2005, NagaZhang2005} on the fictitious domain $\Omega_{i,h}$ for $i =1, 2$.  Then $G_h^i$ is a linear operator from $V_{i,h}$ to $V_{i,h}\times V_{i,h}$ whose value at each nodal point is obtained by the local least squares fitting using sampling points only located in $\Omega_{i,h}$.
According to \cite{Guohailong2016, ZhangNaga2005, NagaZhang2005},   the gradient recovery operator $G_{h}^i$ is bounded in the sense that
\begin{equation}\label{eq:pprbd}
\|G_h^i v_{i,h}\|_{0,\Omega_{i,h}}\lesssim |v_{i,h}|_{1,\Omega_{i,h}} ,\quad v_{i,h}\in V_{i,h},
\end{equation}
and is consistent in following sense that
\begin{equation}\label{eq:pprcs}
\| \nabla v_i-G_h I_{i,h}v_{i}\|_{0,\Omega}\lesssim h^2\|v_i\|_{3,\Omega_{i,h}},  \forall v_i \in H^3(\Omega_{i,h}),
\end{equation}
for $i =1, 2$.

Let  $u_h$ be the finite element solution of the discrete variational problem \eqref{eq:var}.  We define the recovered gradient of $u_h$ as
\begin{equation}\label{eq:uppr}
R_h u_h = \left( G_h^1u_{1,h}, G_h^2u_{2,h}\right).
\end{equation}
The linearity of $G_h^i$ implies $R_h$ is a linear operator from $V_h$ to $V_h\times V_h$.  $R_h$ is called the unfitted polynomial preserving recovery (UPPR).

\begin{remark}
The definition of the gradient recovery operator can be presented in a more general form.  In fact, $G_h^i$ can be chosen as any local
gradient recovery operators \cite{Zhang2007} like simple averaging, weight averaging, SPR and PPR.  For simplicity and efficiency, we only consider
$G_h^i$ as PPR here.
\end{remark}

\begin{remark}
The main idea is to use the standard PPR on each fictitious domain $\Omega_{i,h}$, which is similar to the
improved polynomial preserving recovery for the finite element method based on body-fitted meshes \cite{GuoYang2016}.  But the proposed method \eqref{eq:uppr} does not require
the mesh fitting the interface, and that is why we call it unfitted polynomial preserving recovery method.
\end{remark}

\begin{remark}
We considered the gradient recovery technique for immersed finite element methods in  \cite{GuoYang2017},  which is also
based on  unfitted meshes.  The gradient recovery technique in \cite{GuoYang2017} needs to generate a local body-fitted mesh by dividing every  interface triangle into three sub-triangles, which can lead to skinny triangles and therefore a loss of accuracy.  The proposed gradient recovery operator \eqref{eq:uppr} overcomes this drawback.
\end{remark}

Note that as a function in $V_h\times V_h$, $R_hu_h$ is continuous on each subdomain $\Omega_i$ and is discontinuous in the whole domain $\Omega$ which approximates the
exact gradient $\nabla u$.   Also, similar to  the finite element solution $u_h$,  both $G_h^1u_{1,h}$ and $G_h^2u_{2,h}$  in \eqref{eq:uppr} are, in general, non-zero on  interface triangles $T\in \mathcal{T}_{h, \Gamma}$. For the gradient recovery operator $R_h$ \eqref{eq:uppr}, we can show that it is consistent as follows:
\begin{theorem}\label{thm:upprcs}
Let $R_h:V_h\rightarrow V_h\times V_h$ be the unfitted polynomial preserving recovery operator defined in \eqref{eq:uppr} and $I_h^{\ast}$ be the interpolation
of $u$ into the finite element space $V_h$ as  defined in \eqref{eq:feminterp}. If $u\in H^3(\Omega_1\cup\Omega_2)$,  then we have
 \begin{equation}\label{eq:upprcs}
\|\tilde{\nabla} u - R_hI_h^{\ast}u\|_{0, \Omega_1\cup\Omega_2} \le Ch^2|u|_{3, \Omega_1\cup\Omega_2},
\end{equation}
with $\tilde{\nabla} u = (\nabla u_1, \nabla u_2)$.
\end{theorem}
\begin{proof}
 By  \eqref{eq:pnorm}, \eqref{eq:feminterp}, \eqref{eq:feminterpp}, and  \eqref{eq:pprcs}, we have
 \begin{equation}\label{eq:scproof}
\begin{split}
&\|\tilde{\nabla} u - R_hI_h^{\ast}u\|_{0, \Omega_1\cup\Omega_2} ^2\\
 =&  \|\nabla u_1 - G_h^1I_{1,h}^{\ast}u\|_{0, \Omega_1} ^2 + \|\nabla u_2 - G_h^2I_{2,h}^{\ast}u\|_{0, \Omega_2} ^2\\
 =&  \|\nabla E_{1}u_1 - G_h^1I_{1,h}^{\ast}u\|_{0, \Omega_1} ^2 + \|\nabla E_{2}u_2 - G_h^2I_{2,h}^{\ast}u\|_{0, \Omega_2} ^2\\
\le& \|\nabla E_{1}u_1 - G_h^1I_{1,h}^{\ast}u\|_{0, \Omega_{1,h}} ^2 + \|\nabla E_{2}u_2 - G_h^2I_{2,h}^{\ast}u\|_{0, \Omega_{2,h}} ^2\\
= &\|\nabla E_{1}u_1 - G_h^1I_{1,h}E_{1}u\|_{0, \Omega_{1,h}} ^2 + \|\nabla E_{2}u_2 - G_h^2I_{2,h}E_{2}u\|_{0, \Omega_{2,h}} ^2\\
\le &C_1 h^4 |E_{1}u_1|_{3,\Omega_{1,h}}^2 + C_2h^4 |E_{2,h}u_2|_{3,\Omega_{2,h}}^2\\
\le & C_1 h^4|u_1|_{3,\Omega_1}^2  + C_2 h^4|u_2|_{3,\Omega_2}^2 \\
\le & Ch^4|u|_{3, \Omega_1\cup\Omega_2}^2.
\end{split}
\end{equation}
Taking square root on both sides of \eqref{eq:scproof} completes our proof.
\end{proof}

\begin{remark}
 Theorem \ref{thm:upprcs} means the recovered gradient using the interpolation of the exact solution is superconvergent to the exact gradient at a rate of
 $\mathcal{O}(h^2)$. It is similar to the classical PPR operator for regular elliptic problems.
\end{remark}

\subsection{Superconvergence analysis}  In the following, we shall show the superconvergence property of the proposed UPPR.
Our main superconvergent tool is the supercloseness result provided in Section \ref{sec:superclose}.
\begin{theorem}\label{thm:super}
 Under the same hypothesis as in Theorem \ref{thm:supercloseness}.  We further assume that
$u_h$ is the finite element solution of the discrete variational problem \eqref{eq:var}. Then we have
\begin{equation*}
\|\tilde{\nabla} u - R_hu_h\|_{0, \Omega_1\cup\Omega_2}\le C\left(h^{1+\rho}(\|u\|_{3, \Omega^+\cup\Omega^-} + \|u\|_{2, \infty, \Omega^+\cup\Omega^-}) + h^{3/2} \|u\|_{2, \infty, \Omega^+\cup\Omega^-}\right),
\end{equation*}
where $\tilde{\nabla} u = (\nabla u_1, \nabla u_2)$ and $\rho = \min(\alpha, \frac{\sigma}{2}, \frac{1}{2})$.
\end{theorem}
\begin{proof}
 By the triangle inequality, we have
 \begin{equation*}
\begin{split}
 \|\tilde{\nabla} u - R_hu_h\|_{0, \Omega_1\cup\Omega_2} \le \|\tilde{\nabla} u - R_hI_h^{\ast}u\|_{0, \Omega_1\cup\Omega_2} + \|R_hI_h^{\ast}u- R_hu_h\|_{0, \Omega_1\cup\Omega_2}:= F_1+F_2.
\end{split}
\end{equation*}
According to Theorem \ref{thm:upprcs}, we have
\begin{equation*}
F_1 \le Ch^2|u|_{3, \Omega_1\cup\Omega_2}.
\end{equation*}
For $F_2$, we have
\begin{equation*}
\begin{aligned}
 F_2^2 =& \|G_h^1I_{1,h}^{\ast}u_1- G_h^1u_{1,h}\|_{0, \Omega_1}^2 + \|G_h^2I_{2,h}^{\ast}u_2- G_h^2u_{2,h}\|_{0, \Omega_2}^2\\
   \le&\|G_h^1I_{1,h}^{\ast}u_1- G_h^1u_{1,h}\|_{0, \Omega_{1,h}}^2 + \|G_h^2I_{2,h}^{\ast}u_2- G_h^2u_{2,h}\|_{0, \Omega_{2,h}}^2\\
   \le&C\|\nabla(I_{1,h}^{\ast}u_1- u_{1,h})\|_{0, \Omega_{1,h}}^2 +C \|\nabla(I_{2,h}^{\ast}u_2- u_{2,h})\|_{0, \Omega_{2,h}}^2\\
     \le&C\|\nabla(I_{1,h}^{\ast}u_1- u_{1,h})\|_{0, \Omega_{1}}^2 +C \|\nabla(I_{2,h}^{\ast}u_2- u_{2,h})\|_{0, \Omega_{2}}^2\\
     &+C\|\nabla(I_{1,h}^{\ast}u_1- u_{1,h})\|_{0, \Omega_{\Gamma, h}}^2 +C \|\nabla(I_{2,h}^{\ast}u_2- u_{2,h})\|_{0, \Omega_{\Gamma, h}}^2\\
     = &C\|\nabla(I_{h}^{\ast}u- u_{h})\|_{0, \Omega_{1}\cup\Omega_2}^2 + C\|\nabla(I_{1,h}^{\ast}u_1- u_{1,h})\|_{0, \Omega_{1,h}\setminus\Omega_{1}}^2\\
     & + C \|\nabla(I_{2,h}^{\ast}u_2- u_{2,h})\|_{0, \Omega_{1,h}\setminus\Omega_{2}}^2\\
\end{aligned}
\end{equation*}
\begin{equation*}
\begin{aligned}
      = &C|||u_h-I_h^{\ast}u|||_h^2 + C\|\nabla(I_{1,h}^{\ast}u_1- u_{1,h})\|_{0, \Omega_{1,h}\setminus\Omega_{1}}^2\\
     & +C \|\nabla(I_{2,h}^{\ast}u_2- u_{2,h})\|_{0, \Omega_{1,h}\setminus\Omega_{2}}^2\\
          :=& F_3 + F_4 + F_5.
    \end{aligned}
\end{equation*}
Theorem \ref{thm:superinterp} implies that
\begin{equation*}
F_3\le C\left(h^{1+\rho}(\|u\|_{3, \Omega^+\cup\Omega^-} + \|u\|_{2, \infty, \Omega^+\cup\Omega^-}) + h^{3/2} \|u\|_{2, \infty, \Omega^+\cup\Omega^-}\right)^2.
\end{equation*}
Then, we estimate $F_4$ as
\begin{equation*}
\begin{split}
F_4 \le &C\|\nabla(I_{1,h}^{\ast}u_1- u_{1,h})\|_{0, \Omega_{1,h}\setminus\Omega_{1}}^2\\
\le & C\|\nabla(I_{1,h}^{\ast}u_1- u_{1,h})\|_{0,\Omega_{\Gamma,h}}^2\\
\le & C\|\nabla(I_{1,h}E_1u_1- u_{1,h})\|_{0,\Omega_{\Gamma,h}}^2\\
= &C\sum_{T\in \mathcal{T}_{\Gamma,h}}\|\nabla(I_{1,h}E_1u_1- u_{1,h})\|_{0, T}^2\\
\le &C\sum_{T\in \mathcal{T}_{\Gamma,h}}h^4|E_1u_1|_{2, \infty, T}\\
\le &  Ch^4|u|_{2,\infty, \Omega_1\cup\Omega_2}\sum_{T\in \mathcal{T}_{\Gamma,h}}1\\
\le & Ch^3|u|_{2,\infty, \Omega_1\cup\Omega_2},
\end{split}
\end{equation*}
where we have used the fact $ \sum_{T\in \mathcal{T}_{\Gamma,h}} 1 \approx \mathcal{O}(h^{-1})$.  Similarly, we have
\begin{equation*}
F_5 \le  Ch^3|u|_{2,\infty, \Omega_1\cup\Omega_2}.
\end{equation*}
Combining the estimates for $F_3$, $F_4$, and $F_5$, we have
\begin{equation*}
F_2 \le C\left(h^{1+\rho}(\|u\|_{3, \Omega^+\cup\Omega^-} + \|u\|_{2, \infty, \Omega^+\cup\Omega^-}) + h^{3/2} \|u\|_{2, \infty, \Omega^+\cup\Omega^-}\right),
\end{equation*}
which completes the proof.
\end{proof}

By the above superconvergence result,  we naturally define  a local {\it a posteriori} error estimator on an element $T\in \mathcal{T}_h$ :
\begin{equation}\label{eq:localind}
\eta_T =
    \|\beta^{1/2}(R_hu_h - \nabla u_h)\|_{0, T},
\end{equation}
and the corresponding global error estimator
\begin{equation}\label{eq:globalind}
\eta_h = \left( \sum_{T\in \mathcal{T}_h}\eta_T^2\right)^{1/2}.
\end{equation}

Theorem \ref{thm:super} implies the error estimator \eqref{eq:localind} (or \eqref{eq:globalind})  is   asymptotically exact for the  Nitsche's method:
\begin{theorem}\label{thm:asyexact}
 Assume the same hypothesis in Theorem \ref{thm:supercloseness} and let $u_h$ be the finite element  solution of  the discrete variational problem
 \eqref{eq:var}.  Further assume that there is a constant $C(u)>0$ such that
 \begin{equation}\label{equ:satassum}
 \|\tilde{\nabla}u - \nabla u_h\|_{0,\Omega} \ge C(u) h,
\end{equation}
then it holds that
\begin{equation}
\left | \frac{\eta_h}{\| \beta^{1/2}(\tilde{\nabla} u-\nabla u_h)\|_{0,\Omega}}  -1 \right | \le Ch^{\rho},
\end{equation}
where  $\tilde{\nabla} u = (\nabla u_1, \nabla u_2)$ and $\rho = \min(\alpha, \frac{\sigma}{2}, \frac{1}{2})$.
\end{theorem}
\begin{proof}
 By Theorem \ref{thm:super} and   \eqref{equ:satassum},   we have
 \begin{equation}
\left |  \frac{\eta_h}{\beta^{1/2}\|\tilde{\nabla}u - \nabla u_h\|_{0,\Omega}}  -1 \right |
\le \left |  \frac{ \|\beta^{1/2}(R_hu_h - \tilde{\nabla} u)\|_{0, \Omega}}{\|\beta^{1/2}(\tilde{\nabla}u - \nabla u_h)  \|_{0,\Omega}} \right|
 \le Ch^{\rho}.
\end{equation}
 \end{proof}

\begin{remark}
 For interface problems, there are two types of errors: the  error introduced by geometric discretization and the error introduced by the singularity of the solution.
 The first type of error can be predicted by the curvature of the interface \cite{Chen2014, Zheng2016}.  The error estimator  \eqref{eq:localind} or \eqref{eq:globalind} can be used to estimate the second type of error.
\end{remark}

\section{Numerical examples}
In this section, we show the performance of proposed unfitted polynomial preserving recovery (UPPR) method by several numerical examples with both  simple and complex interface geometries.
The computational domains of  all examples are chosen as  $\Omega=(-1,1)\times(-1,1)$.    For the first two numerical examples,  the uniform triangulations of $\Omega$ are obtained by dividing $\Omega$ into $N^2$ sub-squares and then dividing each sub-square into two right triangles.
The resulted uniform mesh size is $h=\frac{2}{N}$.
 For convenience, we use the following errors in all the examples:
\begin{equation*}
\begin{aligned}
& De:=\|\tilde{\nabla}u - \nabla u_h\|_{0,\Omega_1\cup\Omega_2},\quad D^ie:=\|\nabla u_I- \nabla u_h\|_{0,\Omega_1\cup\Omega_2},\\
& D^re:=\|\tilde{\nabla} u-R_hu_h\|_{0, \Omega_1\cup\Omega_2},
\end{aligned}
 \end{equation*}
with $\tilde{\nabla} u = (\nabla u_1, \nabla u_2)$.

{\bf Example 5.1.}   In this example, we  consider the  interface problem  \eqref{eq:interface}  with homogeneous
jump condition as in \cite{LiLinWu2003}.
The interface is a circular interface of radius $r_0 = 0.5$.
 The exact solution is
\begin{equation*}
u(x,y) =
\left\{
\begin{array}{ll}
    \frac{r^3}{\beta_1}   &  \text{if }   (x,y)\in \Omega_1, \\
      \frac{r^3}{\beta_2} + \left( \frac{1}{\beta^-}-\frac{1}{\beta^+} \right)r_0^3&  \text{if } (x,y)\in \Omega_2,\\
   \end{array}
\right.
\end{equation*}
where $r = \sqrt{x^2+y^2}$.

We consider the  following  four typical different
jump ratios:  $\beta_1/\beta_2 = 1/10$ (moderate jump), $\beta_1/\beta_2 = 1/1000$ (large jump), $\beta_1/\beta_2 = 1/100000$ (huge jump),
and $\beta_1/\beta_2 = 100000$ (huge jump).   The numerical errors are displayed in Tables \ref{tab:ex51a}-\ref{tab:ex51d}.
We observe an optimal convergence in the $H^1$-seminorm as predicted by  Theorem \ref{thm:optimal}.
The observed $\mathcal{O}(h^{1.5})$ supercloseness and superconvergence   confirm our theoretical results.
In addition, we observe the same superconvergence results in all different cases.
It means that the superconvergence results are independent of the jump ratio  of the coefficient.
In Figure \ref{fig:circle}, we plot the recovered gradient on the initial mesh.

\begin{table}[htb!]
\centering
\caption{Numerical results for Example 5.1 with $\beta_1=10, \beta_2=1$. }\label{tab:ex51a}
 \begin{tabular}{|c|c|c|c|c|c|c|c|}
\hline
 $h$ & $De$ & order& $D^{i}e$ & order& $D^{r}_re$ & order\\ \hline\hline
 1/16 &4.61e-02&--&2.37e-02&--&1.82e-02&--\\ \hline
 1/32 &2.34e-02&0.98&9.34e-03&1.34&7.70e-03&1.25\\ \hline
 1/64 &1.17e-02&1.00&3.28e-03&1.51&2.75e-03&1.48\\ \hline
 1/128 &5.88e-03&1.00&1.17e-03&1.48&9.95e-04&1.47\\ \hline
 1/256 &2.94e-03&1.00&4.08e-04&1.52&3.36e-04&1.56\\ \hline
 1/512 &1.47e-03&1.00&1.43e-04&1.51&1.17e-04&1.53\\ \hline
 1/1024 &7.35e-04&1.00&5.08e-05&1.49&4.17e-05&1.48\\ \hline
\end{tabular}
\end{table}

\begin{table}[htb!]
\centering
\caption{Numerical results for Example 5.1 with $\beta_1=1000, \beta_2=1$. }\label{tab:ex51b}
 \begin{tabular}{|c|c|c|c|c|c|c|c|}
\hline
 $h$ & $De$ & order& $D^{i}e$ & order& $D^{r}_re$ & order\\ \hline\hline
 1/16 &4.19e-02&--&2.62e-02&--&2.15e-02&--\\ \hline
 1/32 &2.13e-02&0.98&9.98e-03&1.39&8.52e-03&1.33\\ \hline
 1/64 &1.06e-02&1.00&3.53e-03&1.50&3.09e-03&1.46\\ \hline
 1/128 &5.33e-03&1.00&1.25e-03&1.50&1.12e-03&1.47\\ \hline
 1/256 &2.66e-03&1.00&4.33e-04&1.52&3.75e-04&1.57\\ \hline
 1/512 &1.33e-03&1.00&1.52e-04&1.51&1.29e-04&1.54\\ \hline
 1/1024 &6.66e-04&1.00&5.41e-05&1.49&4.59e-05&1.49\\ \hline
\end{tabular}
\end{table}

\begin{table}[htb!]
\centering
\caption{Numerical results for Example 5.1 with $\beta_1=1, \beta_2=100000$. }\label{tab:ex51c}
\begin{tabular}{|c|c|c|c|c|c|c|c|}
\hline
 $h$ & $De$ & order& $D^{i}e$ & order& $D^{r}_re$ & order\\ \hline\hline
 1/16 &1.99e-01&--&2.95e-02&--&3.23e-02&--\\ \hline
 1/32 &9.97e-02&1.00&9.94e-03&1.57&1.06e-02&1.61\\ \hline
 1/64 &4.98e-02&1.00&3.53e-03&1.50&3.08e-03&1.78\\ \hline
 1/128 &2.49e-02&1.00&1.19e-03&1.56&1.05e-03&1.55\\ \hline
 1/256 &1.25e-02&1.00&4.33e-04&1.46&3.85e-04&1.45\\ \hline
 1/512 &6.23e-03&1.00&1.56e-04&1.47&1.38e-04&1.48\\ \hline
 1/1024 &3.12e-03&1.00&5.51e-05&1.50&4.85e-05&1.51\\ \hline
\end{tabular}
\end{table}

\begin{table}[htb!]
\centering
\caption{Numerical results for Example 5.1 with $\beta_1=100000, \beta_2=1$. }\label{tab:ex51d}
\begin{tabular}{|c|c|c|c|c|c|c|c|}
\hline
 $h$ & $De$ & order& $D^{i}e$ & order& $D^{r}_re$ & order\\ \hline\hline
 1/16 &4.19e-02&--&2.62e-02&--&2.15e-02&--\\ \hline
 1/32 &2.13e-02&0.98&9.99e-03&1.39&8.54e-03&1.33\\ \hline
 1/64 &1.06e-02&1.00&3.54e-03&1.50&3.10e-03&1.46\\ \hline
 1/128 &5.33e-03&1.00&1.25e-03&1.50&1.12e-03&1.46\\ \hline
 1/256 &2.66e-03&1.00&4.33e-04&1.52&3.84e-04&1.55\\ \hline
 1/512 &1.33e-03&1.00&1.52e-04&1.51&1.37e-04&1.48\\ \hline
 1/1024 &6.66e-04&1.00&5.41e-05&1.49&4.65e-05&1.56\\ \hline
\end{tabular}
\end{table}

\begin{figure}[h b p]
   \centering
     \subcaptionbox{\label{fig:circle_rx}}
  {\includegraphics[width=0.45\textwidth]{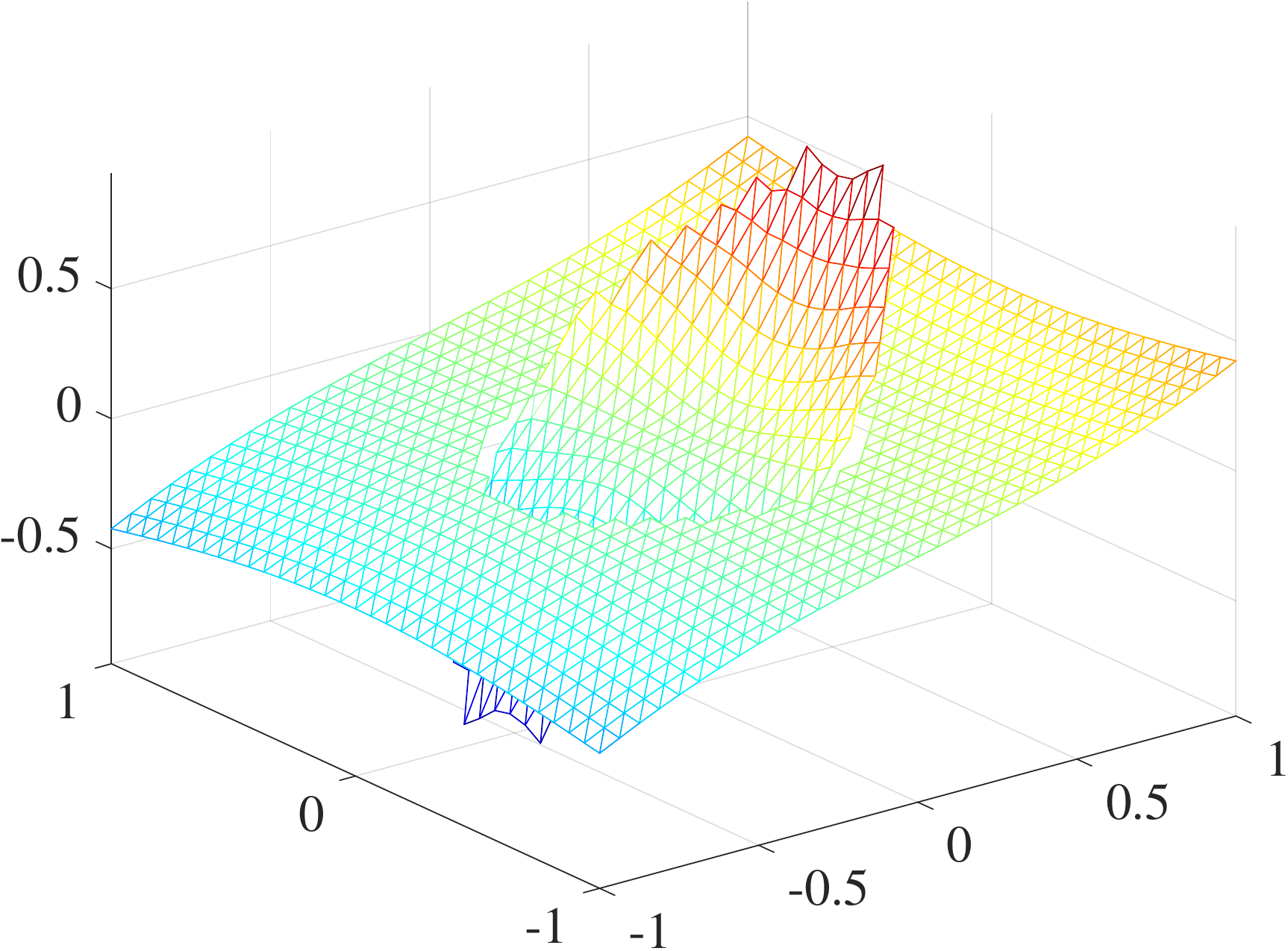}}
  \subcaptionbox{\label{fig:circle_ry}}
   {\includegraphics[width=0.45\textwidth]{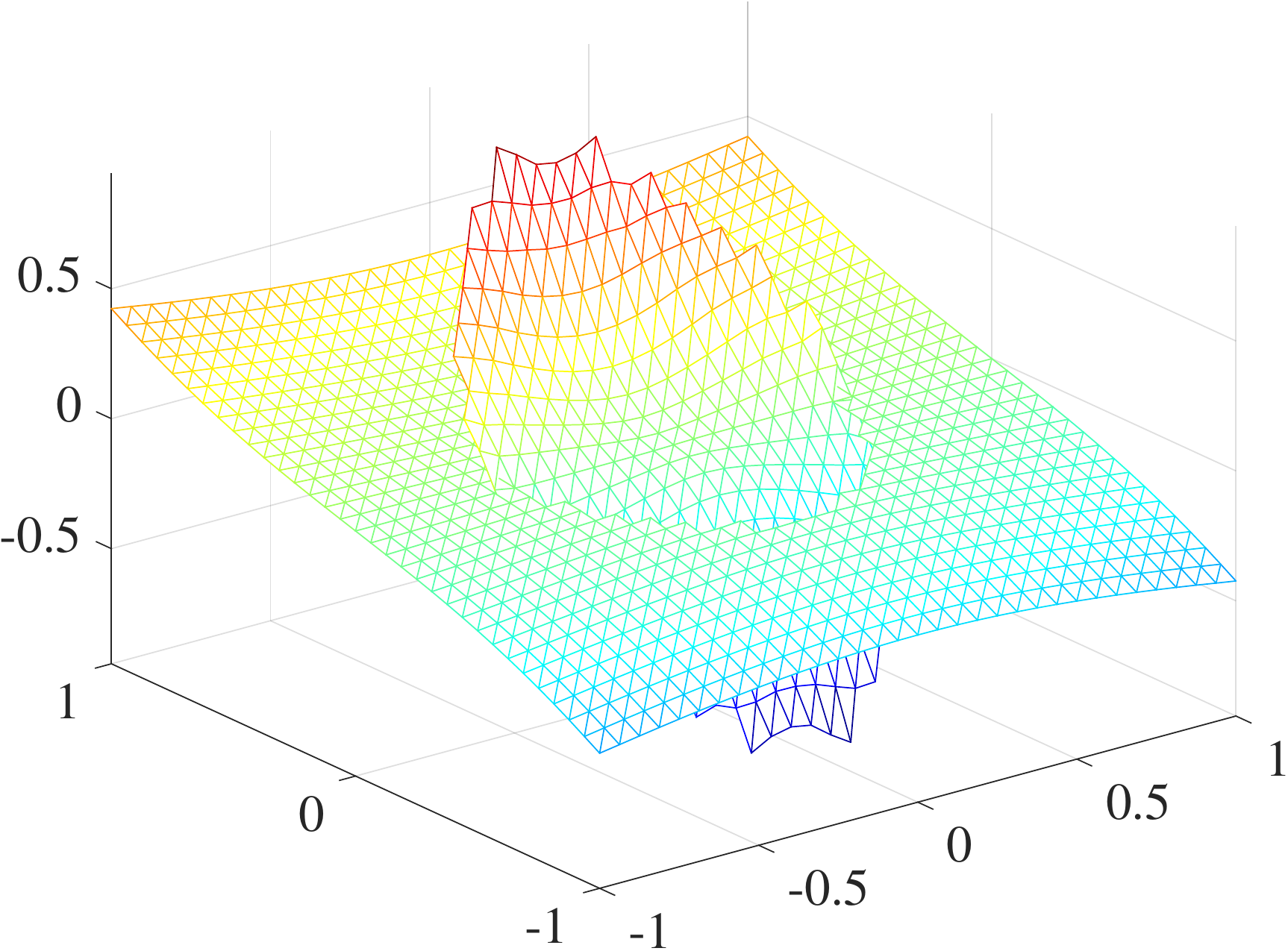}}
   \caption{Plots of recovered gradient for Example 1. (a): $x$-component ; (b): $y$-component.}\label{fig:circle}
\end{figure}

{\bf Example 5.2.}  In this example,  we consider the flower-shape interface problem with non-homogeneous jump conditions as studied in \cite{MuWang2013, ZhouWei2006}.
The interface curve $\Gamma$ in polar coordinates is given by
\begin{equation*}
 r = \frac{1}{2} + \frac{\sin(5\theta)}{7}.
\end{equation*}
It contains both convex and concave parts, as demonstrated  in Figure \ref{fig:flower_lst}.
The diffusion coefficient is piecewise constant with $\beta_1=1$ and $\beta_2=10$.   The  right-hand side function $f$ in \eqref{eq:model} is chosen to match the exact solution
\begin{equation*}
u(x,y) =
\left\{
\begin{array}{ll}
    e^{(x^2+y^2)}, & \text{if } (x,y)\in \Omega_1\\
    0.1(x^2+y^2)^2-0.01\ln(2\sqrt{x^2+y^2}),&  \text{if } (x,y)\in \Omega_2,\\
   \end{array}
\right.
\end{equation*}
and the jump conditions  \eqref{eq:valuejump}-\eqref{eq:fluxjump} are  provided by the exact solution.

In Figure \ref{fig:flower_sol},  we plot the numerical solution on the initial mesh which clearly indicates the non-homogeneous jump in function value.
We show the numerical results in Table \ref{tab:ex52}.  As expected, we observe the first-order convergence for the gradient of finite element solution.
For the recovered gradient, $\mathcal{O}(h^{1.5})$ convergence is observed, which is in agreement with Theorem  \ref{thm:super}.
The recovered gradient on the initial mesh is visualized in Figure \ref{fig:flower}.

\begin{table}[htb!]
\centering
\caption{Numerical results for Example 5.2. }\label{tab:ex52}
\begin{tabular}{|c|c|c|c|c|c|c|c|}
\hline
 $h$ & $De$ & order& $D^{i}e$ & order& $D^{r}_re$ & order\\ \hline\hline
 1/16 &8.86e-02&--&5.81e-02&--&3.74e-02&--\\ \hline
 1/32 &3.90e-02&1.19&1.50e-02&1.95&1.19e-02&1.65\\ \hline
 1/64 &1.90e-02&1.04&4.37e-03&1.78&3.57e-03&1.74\\ \hline
 1/128 &9.48e-03&1.00&1.57e-03&1.48&1.29e-03&1.47\\ \hline
 1/256 &4.74e-03&1.00&5.63e-04&1.48&4.72e-04&1.45\\ \hline
 1/512 &2.37e-03&1.00&2.00e-04&1.50&1.71e-04&1.47\\ \hline
 1/1024 &1.18e-03&1.00&7.06e-05&1.50&6.62e-05&1.37\\ \hline
\end{tabular}
\end{table}

\begin{figure}
   \centering
   \subcaptionbox{\label{fig:flower_lst}}
  {\includegraphics[width=0.4\textwidth]{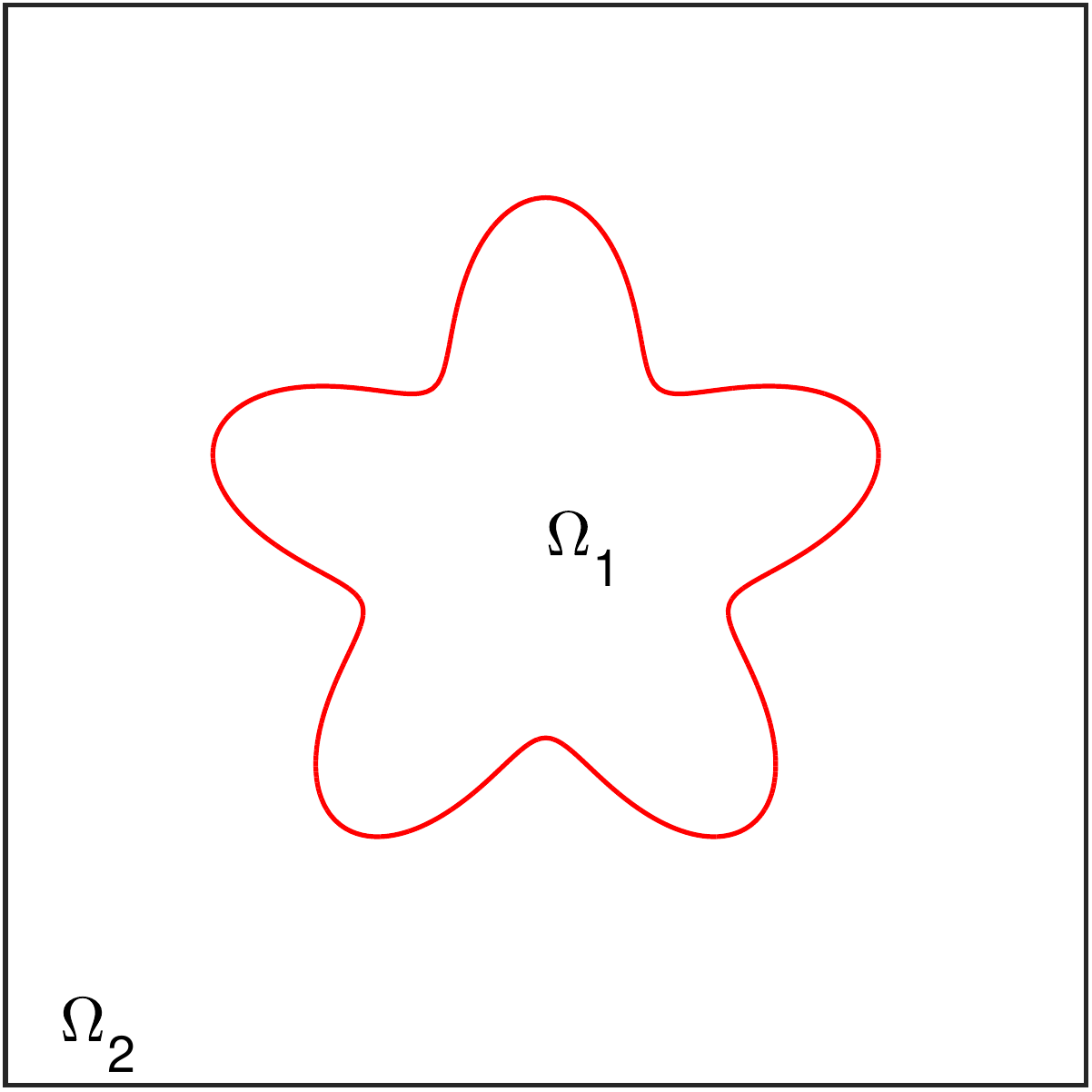}}
  \subcaptionbox{\label{fig:flower_sol}}
   {\includegraphics[width=0.55\textwidth]{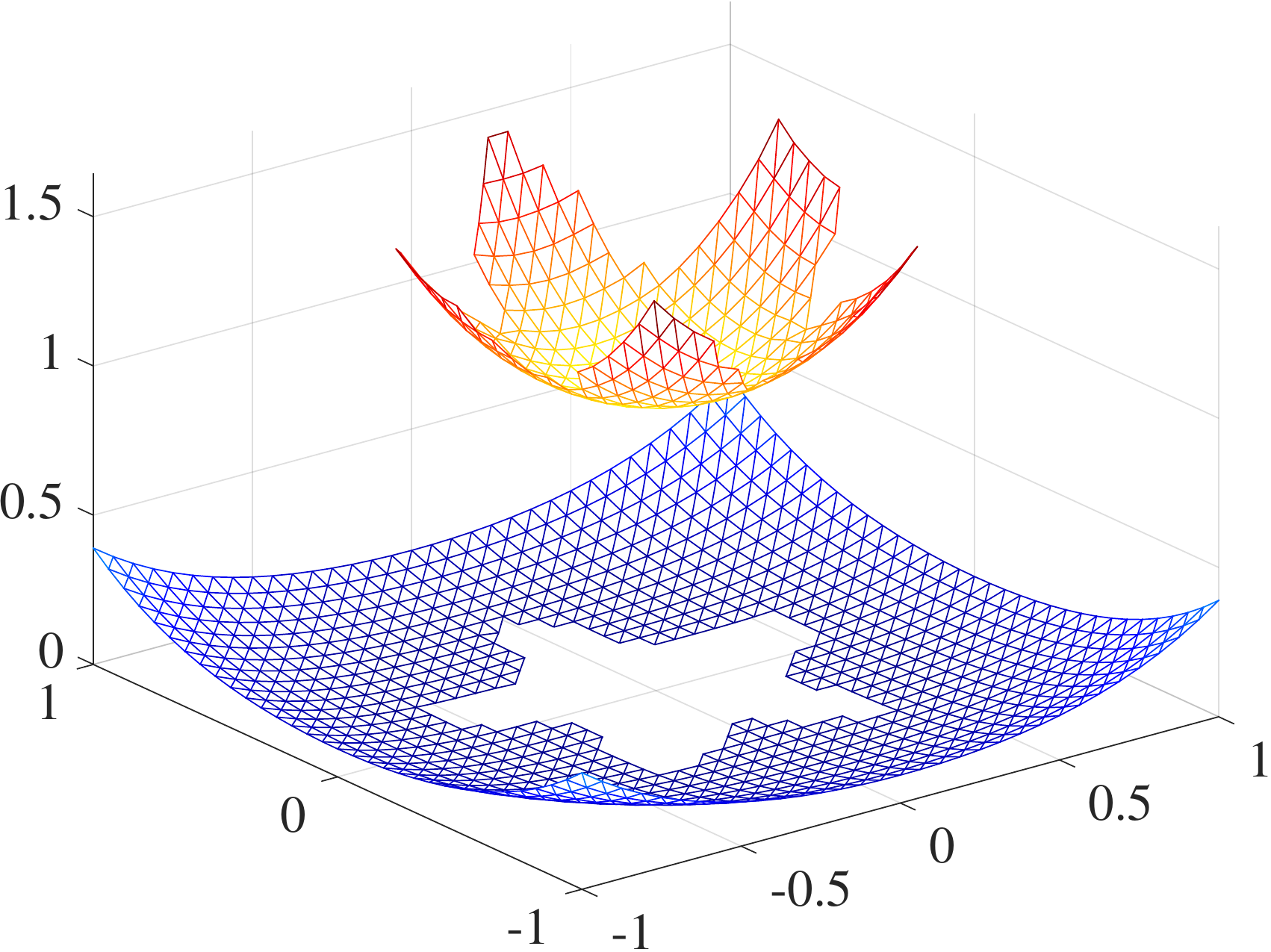}}
   \caption{Plots for Example 4. (a): Plot of the interface; (b): Plot of numerical solution.}
\end{figure}

\begin{figure}
   \centering
     \subcaptionbox{\label{fig:flower_rx}}
  {\includegraphics[width=0.47\textwidth]{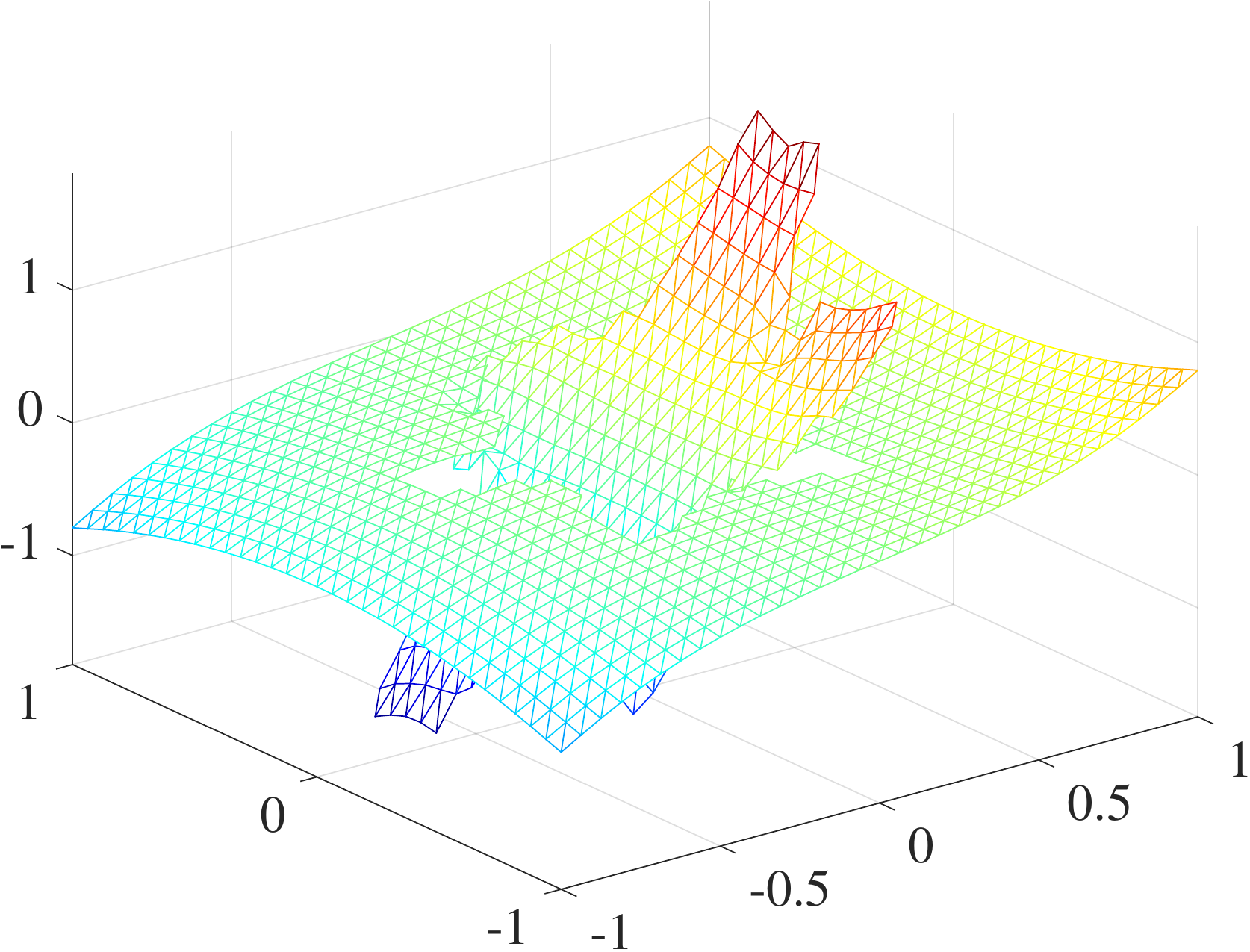}}
  \subcaptionbox{\label{fig:flower_ry}}
   {\includegraphics[width=0.47\textwidth]{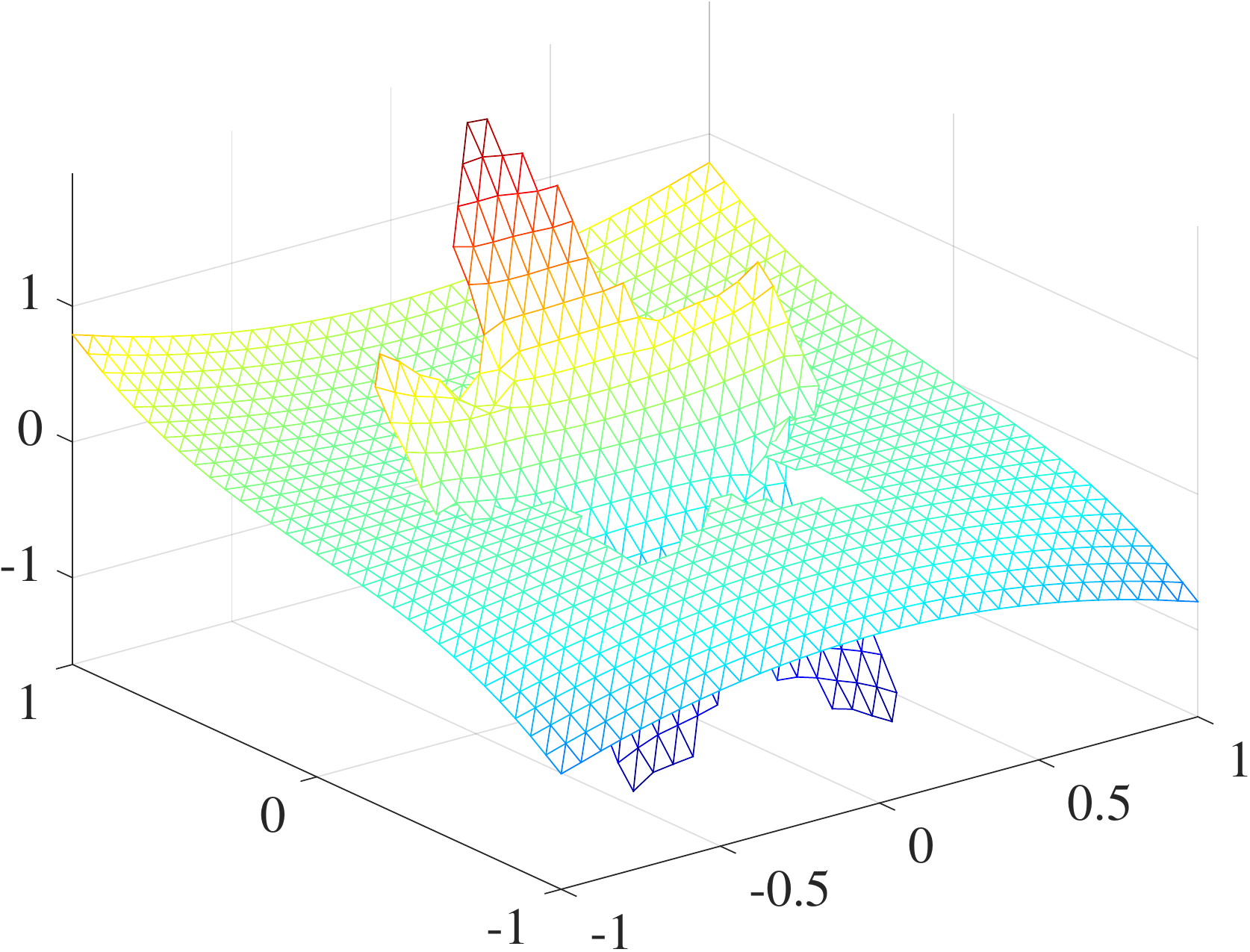}}
   \caption{Plots of recovered gradient for Example 2 on the initial mesh. (a): $x$-component ; (b): $y$-component.}\label{fig:flower}
\end{figure}

{\bf Example 5.3.}  In this example, we consider the interface problem with complex geometrical structure as in \cite{MuWang2015}.  The interface in polar coordinates  is given by
\begin{equation*}
r = 0.40178(1+\cos(2\theta)\sin(6\theta))\cos(\theta).
\end{equation*}
The interface and subdomains are plotted in Figure \ref{fig:flower_lst}.
The coefficient function is
\[
\beta(x,y) =
\left\{
\begin{array}{ll}
(x^2-y^2-7)/7,& (x,y)\in \Omega_1,\\
(xy+2)/5,& (x,y)\in \Omega_2;
\end{array}
\right.
\]
and the exact function is
\[
u(x,y) =
\left\{
\begin{array}{ll}
\sin(x+y)+\cos(x+y)+1,& (x,y)\in \Omega_1,\\
x+y+1,&(x,y)\in \Omega_2.
\end{array}
\right.
\]

As plotted in Figure \ref{fig:linmu_lst}, the interface contains complex geometrical structure.
To guarantee the Assumption \ref{ass:inter}, we need an extremely fine mesh.  It would increase the computational cost.
To reduce the computational cost, we propose an adaptive strategy   to generate  an initial unfitted mesh.
Here we use the curvature-based {\it a posterior} estimator to guide the refinement of the mesh as in \cite{Chen2014}.
Different from the mesh generated in \cite{Chen2014}, the resulted mesh is an unfitted mesh and all triangles are
perfect right triangles.

Figure \ref{fig:linmu_msh} plots the generated initial unfitted mesh. It is easy to see that the mesh is refined around the part of the interface with
high curvature.  The other four levels of unfitted meshes are obtained by uniform refinement.  The numerical results are summarized in Table \ref{tab:ex53}.
Note that in Table \ref{tab:ex53}, convergence rates are listed  with respect to the degree of freedom (DOF).
  The
corresponding  convergent rates
with respect to the mesh size $h$ are double of what we present in Table \ref{tab:ex53}.  The gradient of finite element solution converges to
the exact gradient at the rate of $\mathcal{O}(h)$ while the recovered gradient superconverges  at the rate of $\mathcal{O}(h^{1.5})$.
Additionally, the predicted supercloseness is observed in the numerical experiment.

\begin{figure}
   \centering
     \subcaptionbox{\label{fig:linmu_lst}}
  {\includegraphics[width=0.45\textwidth]{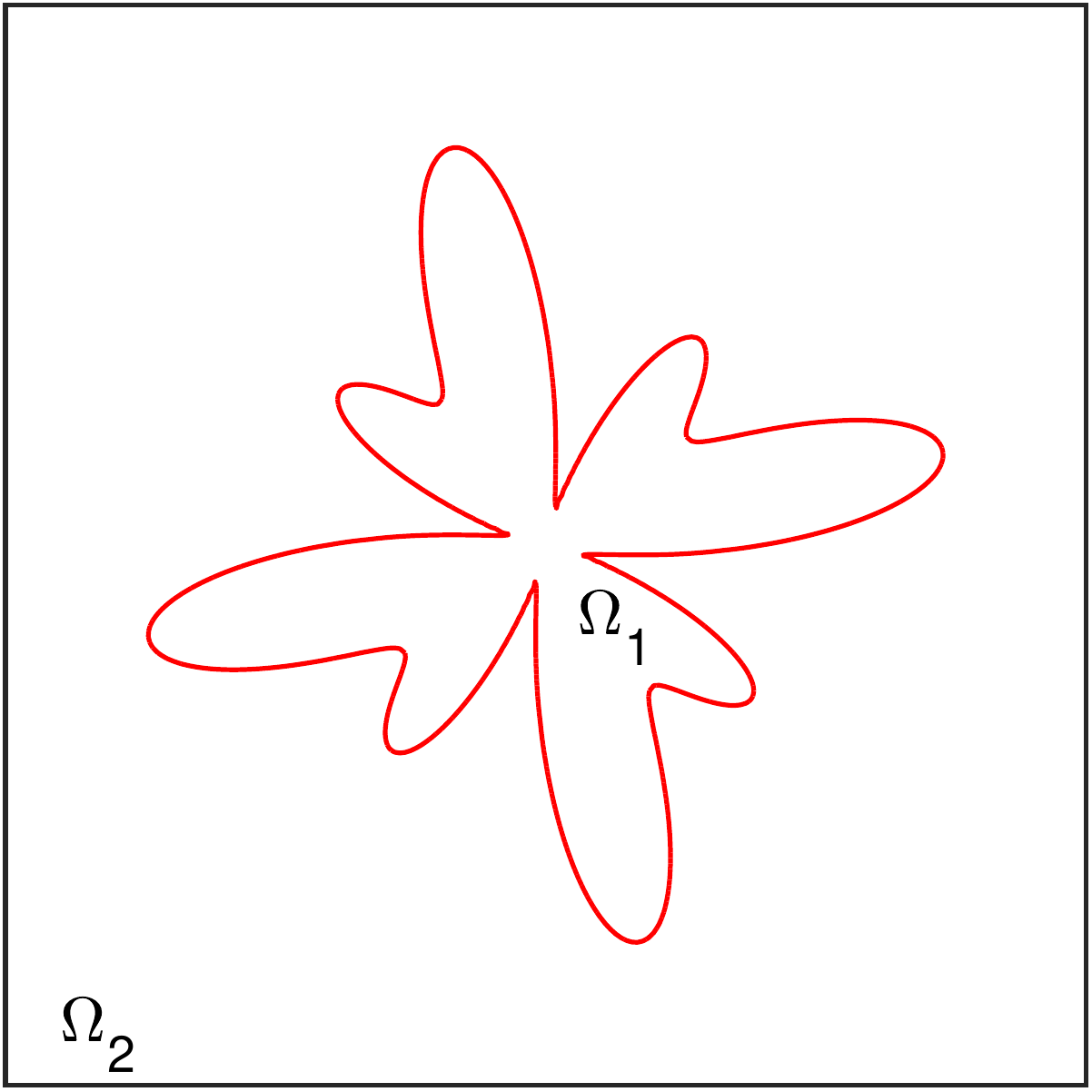}}
  \subcaptionbox{\label{fig:linmu_msh}}
   {\includegraphics[width=0.45\textwidth]{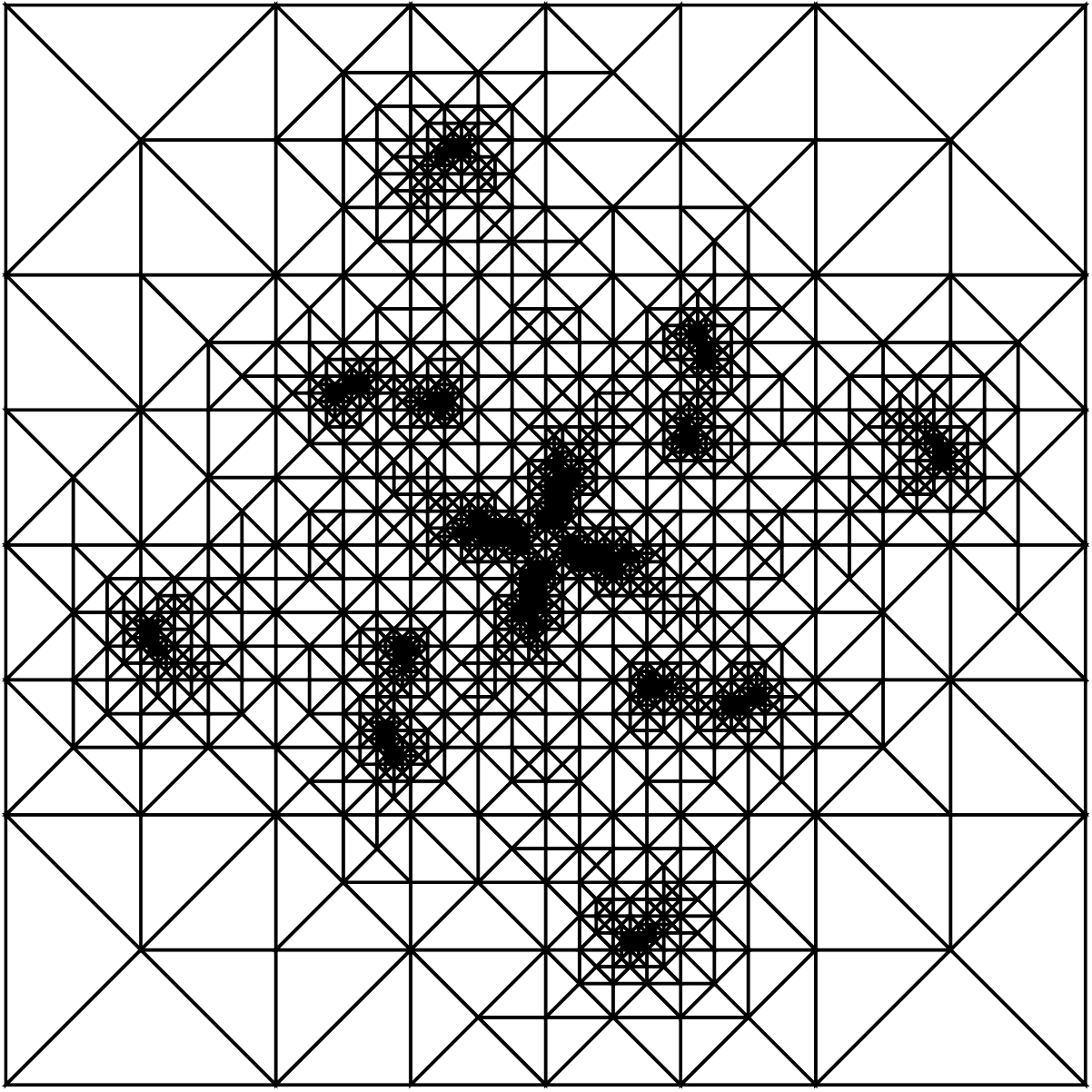}}
   \caption{Plots for Example 3. (a): Plot of the interface; (b): Initial unfitted mesh;}
\end{figure}

\begin{table}[htb!]
\centering
\caption{Numerical results for Example 5.3. }\label{tab:ex53}
\begin{tabular}{|c|c|c|c|c|c|c|c|}
\hline
DOF & $De$ & order& $D^{i}e$ & order& $D^{r}_re$ & order\\ \hline\hline
 2573 &2.49e-02&--&1.32e-02&--&8.68e-03&--\\ \hline
 10265 &1.29e-02&0.48&4.42e-03&0.79&3.20e-03&0.72\\ \hline
 41009 &6.57e-03&0.49&1.56e-03&0.75&1.25e-03&0.68\\ \hline
 163937 &3.30e-03&0.50&5.09e-04&0.81&4.46e-04&0.74\\ \hline
 655553 &1.66e-03&0.50&1.74e-04&0.77&1.59e-04&0.74\\ \hline
 2621825 &8.28e-04&0.50&5.92e-05&0.78&5.48e-05&0.77\\ \hline
\end{tabular}
\end{table}

{\bf Example 5.4.}  In this example, we consider the interface problem as in \cite{Teran2010,Chen2014}.  The interface  in parametric
form is given by
\[
\left\{
\begin{array}{ll}
 x(t) = r(t) \cos(\theta(t)),\\
  y(t) = r(t) \sin(\theta(t));
\end{array}
\right.
\]
where
\begin{equation*}
\theta(t) = t+\sin(4t), \quad r(t) = 0.60125+0.24012\cos(4t+\pi/2).
\end{equation*}
The coefficient function is
\[
\beta(x,y) =
\left\{
\begin{array}{ll}
4+\sin(x+y),& (x,y)\in \Omega_1,\\
10+x^2+y^2,& (x,y)\in \Omega_2;
\end{array}
\right.
\]
and the exact solution is
\[
u(x,y) =
\left\{
\begin{array}{ll}
\sin(x)\cos(y),& (x,y)\in \Omega_1,\\
1-x^2-y^2,&(x,y)\in \Omega_2.
\end{array}
\right.
\]

\begin{figure}
   \centering
     \subcaptionbox{\label{fig:teran_lst}}
  {\includegraphics[width=0.45\textwidth]{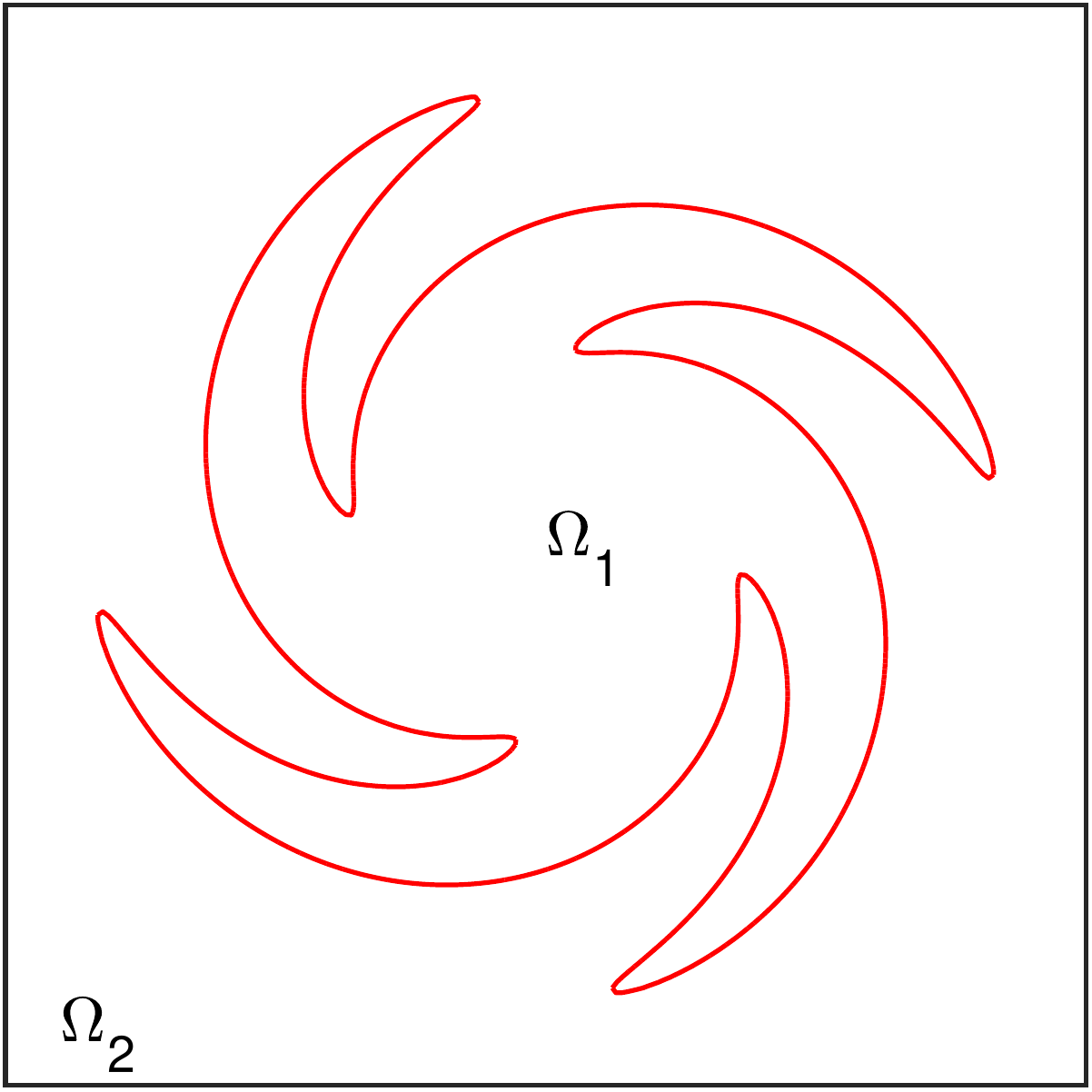}}
  \subcaptionbox{\label{fig:teran_msh}}
   {\includegraphics[width=0.45\textwidth]{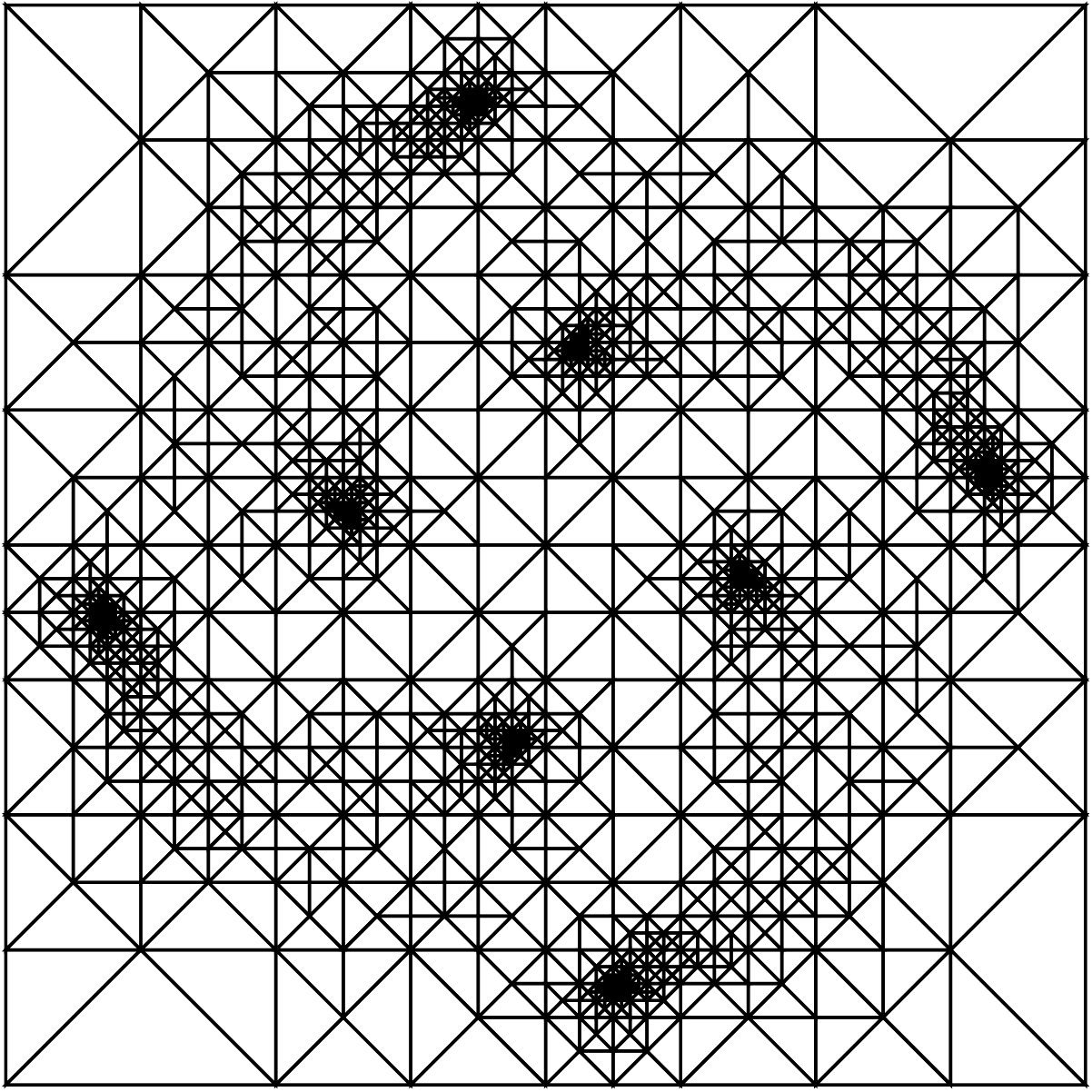}}
   \caption{Plots for Example 4. (a): Plot of the interface; (b): Initial unfitted mesh;}
\end{figure}

\begin{table}[htb!]
\centering
\caption{Numerical results for Example 5.4. }\label{tab:ex54}
\begin{tabular}{|c|c|c|c|c|c|c|c|}
\hline
 DOF & $De$ & order& $D^{i}e$ & order& $D^{r}_re$ & order\\ \hline\hline
 1381 &2.14e-01&--&9.04e-02&--&8.01e-02&--\\ \hline
 5489 &1.12e-01&0.47&2.51e-02&0.93&2.02e-02&1.00\\ \hline
 21889 &5.67e-02&0.49&7.60e-03&0.86&6.65e-03&0.80\\ \hline
 87425 &2.85e-02&0.50&2.32e-03&0.86&2.11e-03&0.83\\ \hline
 349441 &1.42e-02&0.50&7.29e-04&0.83&6.99e-04&0.80\\ \hline
 1397249 &7.12e-03&0.50&2.43e-04&0.79&2.39e-04&0.78\\ \hline
\end{tabular}
\end{table}

The interface $\Gamma$, shown in Figure \ref{fig:teran_lst}, contains complex geometrical structure. We use the same algorithm as in Example 5.3 to generate an initial unfitted mesh which is
plotted in Figure \ref{fig:teran_msh}.  Table \ref{tab:ex54} lists the numerical results. Clearly, we observe the desired optimal convergence and superconvergence rates.

{\bf Example 5.5.}  In this example, we consider the interface problem as in \cite{Li1998a, Chen2014}.
The interface $\Gamma$ in parametric form is defined by
\[
\left\{
\begin{array}{ll}
 x(t) = r(\theta) \cos(\theta)+x_c,\\
  y(t) = r(\theta) \sin(\theta)+y_c;
\end{array}
\right.
\]
where $r(\theta) = r_0+r_1\sin(\omega\theta)$, $0\le\theta<2\pi$.

In this test, we take $r_0=0.4$, $r_1=0.2$, $\omega=20$, and $x_c=y_c=0.02\sqrt{5}$.
The coefficient $\beta$ is a piecewise constant  with $\beta_1 = 1$ and $\beta_2 = 10$.
The exact function is
\[
u(x,y) =
\left\{
\begin{array}{ll}
r^2/\beta_1,& (x,y)\in \Omega_1,\\
(r^4-0.1\log(2r))/\beta_2,&(x,y)\in \Omega_2.
\end{array}
\right.
\]

\begin{figure}
   \centering
     \subcaptionbox{\label{fig:chenlong_lst}}
  {\includegraphics[width=0.45\textwidth]{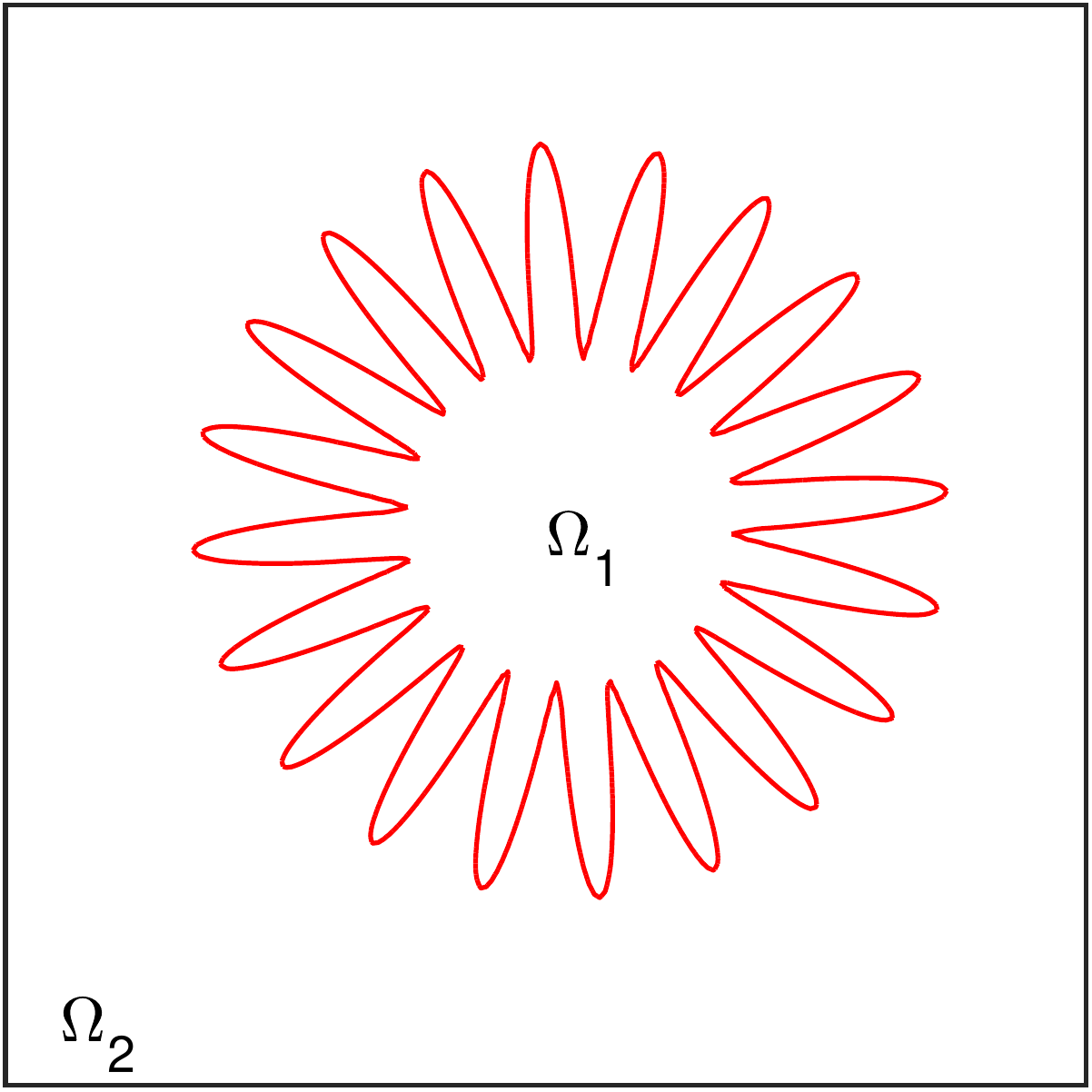}}
  \subcaptionbox{\label{fig:chenlong_msh}}
   {\includegraphics[width=0.45\textwidth]{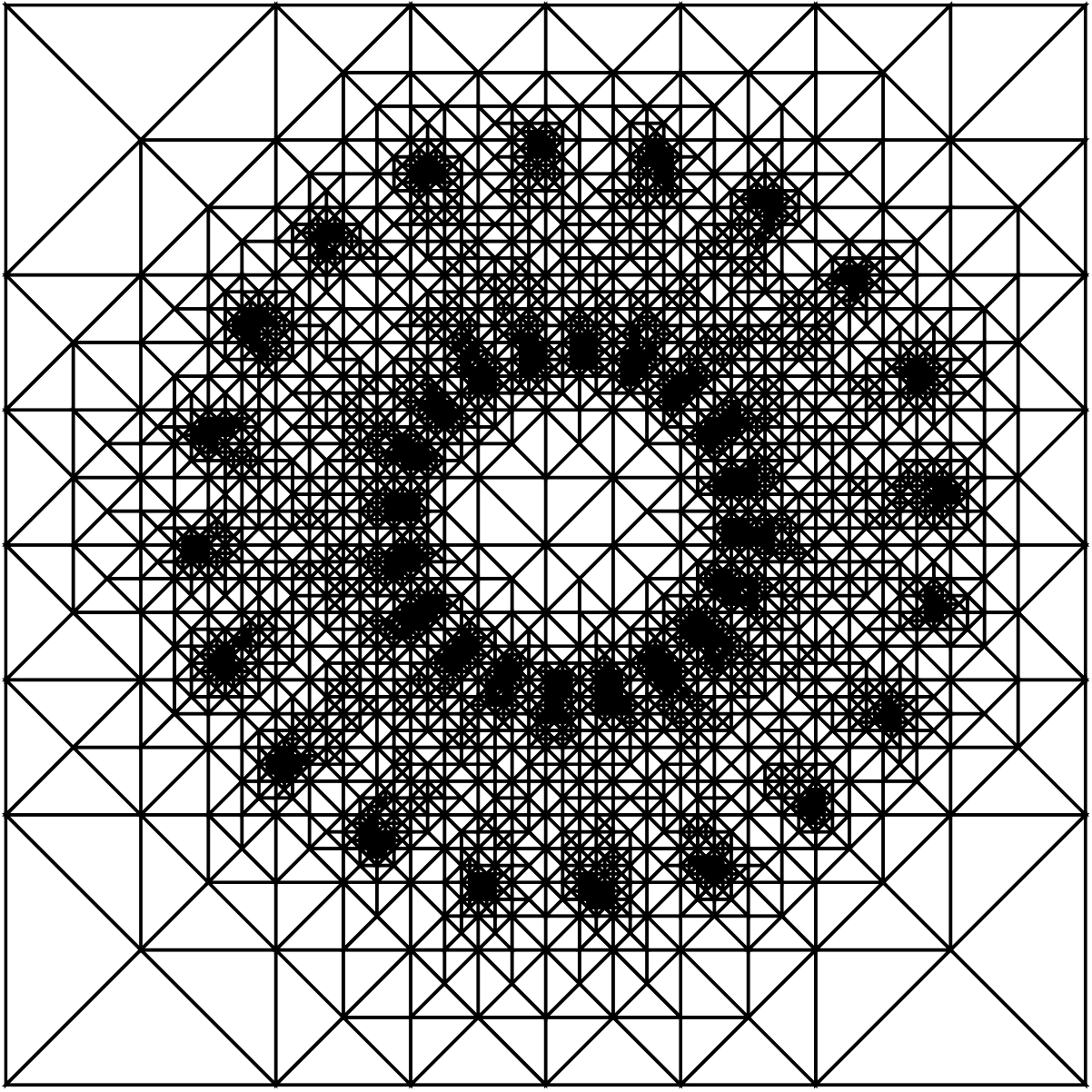}}
   \caption{Plots for Example 5. (a): Plot of the interface; (b): Initial unfitted mesh;}
\end{figure}

\begin{table}[htb!]
\centering
\caption{Numerical results for Example 5.5. }\label{tab:ex55}
\begin{tabular}{|c|c|c|c|c|c|c|c|}
\hline
 DOF & $De$ & order& $D^{i}e$ & order& $D^{r}_re$ & order\\ \hline\hline
 6514 &1.30e-01&--&5.36e-02&--&4.39e-02&--\\ \hline
 26027 &6.70e-02&0.48&1.53e-02&0.91&1.58e-02&0.74\\ \hline
 104053 &3.39e-02&0.49&4.23e-03&0.93&4.74e-03&0.87\\ \hline
 416105 &1.70e-02&0.50&1.17e-03&0.93&1.32e-03&0.92\\ \hline
 1664209 &8.51e-03&0.50&3.26e-04&0.92&3.72e-04&0.91\\ \hline
 6656417 &4.25e-03&0.50&9.26e-05&0.91&9.97e-05&0.95\\ \hline
\end{tabular}
\end{table}

The interface $\Gamma$ is plotted in Figure  \ref{fig:chenlong_lst} and the adaptively refined initial mesh is shown in Figure \ref{fig:chenlong_msh}.
The numerical results are given in Table \ref{tab:ex55}.  The observed results confirm the first-order convergence rate as predicted by Theorem \ref{thm:optimal}.
For the errors $De^i$ and $De^r$,   $\mathcal{O}(h^{1.8})$ order decaying rates can be observed which are better than our theoretical results.
Compared to the numerical results using a body-fitted mesh  in \cite{Chen2014},  we achieve the same accuracy by using an unfitted mesh with about
one sixth of the total mesh grid points.

\section{Conclusion}
In this paper, we propose a new gradient recovery technique based on the Nitsche's method. Compared to our previous works \cite{GuoYang2016,GuoYang2017,GuoYangZhang2017}, it avoids the loss of accuracy of gradient near the interface caused by skinny triangles.
By proving the supercloseness result for the Nitsche's method, we are able to show that the recovered gradient is superconvergent to the exact gradient. As a byproduct, we propose a curvature estimator based adaptive algorithm to generate
initial unfitted triangulations for the elliptic interface problems with complex geometry, which greatly reduces the computational cost as illustrated in Examples 5.3, 5.4 and 5.5.   The future  work is planned in several different directions:  firstly, we will extend the study into three dimension problems; secondly, we will consider other type equations like
elastic interface problems and wave propagation problems in heterogeneous media;  thirdly, we will combine the curvature estimator and the recovery-based
{\it a posterior} error estimator to derive adaptive algorithms for the complex interface problems.

\section*{Acknowledgement}
This work was partially supported by the NSF grant DMS-1418936 and DMS-1107291.

\bibliography{mybibfile}

\end{document}